\documentclass[10pt]{amsart}

\usepackage{amsmath}
\usepackage{amssymb}
\usepackage{ upgreek }
\usepackage{mathtools}
\usepackage{centernot}
\usepackage{stmaryrd}
\usepackage{esvect}
\usepackage{amsthm}
\usepackage{upgreek}
\usepackage{comment}
\usepackage{amsfonts}
\usepackage{physics}
\usepackage{tikz-cd}
\usepackage[utf8]{inputenc}
\usepackage[english]{babel}
\usepackage{graphicx}
\usepackage[margin=1in]{geometry}
\usepackage{mathtools}
\DeclarePairedDelimiter{\ceil}{\lceil}{\rceil}
\graphicspath{{images/}}

\renewcommand{\part}[1]{\noindent\textbf{Part #1)}}

\newcommand{\ba}{\begin{align*}}

\newcommand{\ea}{\end{align*}}

\newcommand{\N}{\mathbb{N}}

\newcommand{\Int}[1]{%
  {\kern0pt#1}^{\mathrm{o}}%
}

\newcommand{\bp}{\begin{pmatrix}}
\newcommand{\ep}{\end{pmatrix}}

\newcommand{\interior}[1]{%
  {\kern0pt#1}^{\mathrm{o}}%
}

\usepackage{xargs}                      
\usepackage{xcolor}
\usepackage[colorinlistoftodos,prependcaption,textsize=tiny]{todonotes}
\newcommandx{\kate}[2][1=]{\todo[linecolor=red,backgroundcolor=red!25,bordercolor=red,#1]{#2}}
\newcommandx{\zhen}[2][1=]{\todo[linecolor=lime,backgroundcolor=lime!25,bordercolor=lime,#1]{#2}}
\newcommandx{\yi}[2][1=]{\todo[linecolor=black,backgroundcolor=black!25,bordercolor=black,#1]{#2}}
\newcommandx{\kabir}[2][1=]{\todo[linecolor=purple,backgroundcolor=purple!25,bordercolor=purple,#1]{#2}}

\newcommandx{\nat}[2][1=]{\todo[linecolor=blue,backgroundcolor=blue!25,bordercolor=blue,#1]{#2}}
\newcommandx{\jonah}[2][1=]{\todo[linecolor=yellow,backgroundcolor=yellow!25,bordercolor=yellow,#1]{#2}}
\newcommandx{\colin}[2][1=]{\todo[linecolor=gray,backgroundcolor=gray!25,bordercolor=gray,#1]{#2}}
\begin{document}
\author{C. Adams, O. Eisenberg, J. Greenberg, K. Kapoor, Z. Liang, \\
K. O'Connor, N. Pacheco-Tallaj, Y. Wang}
\title{TG-Hyperbolicity of Virtual Links}
\begin{abstract}We extend the theory of hyperbolicity of links in the 3-sphere to tg-hyperbolicity of virtual links, using the fact that the theory of virtual links can be translated into the theory of links living in closed orientable thickened surfaces. When the boundary surfaces are taken to be totally geodesic, we obtain a tg-hyperbolic structure with a unique associated volume.   We prove that all virtual alternating links are tg-hyperbolic. We further extend tg-hyperbolicity to several classes of non-alternating virtual links. We then consider bounds on volumes of virtual links and include a table for volumes of the 116 nontrivial virtual knots of four or fewer crossings, all of which, with the exception of the trefoil knot, turn out to be tg-hyperbolic. 
\end{abstract}
\date{March 2019}

\maketitle

\theoremstyle{definition}
\newtheorem{definition}{Definition}[section]
\newtheorem{theorem}{Theorem}[section]
\newtheorem{lemma}[theorem]{Lemma}
\newtheorem{conjecture}{Conjecture}
\newtheorem{corollary}[theorem]{Corollary}
\newtheorem{proposition}[theorem]{Proposition}
\newtheorem{example}{Example}[section]

\section{Introduction}\label{intro}

In 1978, Thurston \cite{thurston} proved the remarkable result that knots in $S^3$ fall into three categories: torus knots, satellite knots and hyperbolic knots. Hyperbolic knots are knots such that their complement $S^3 \setminus L$ admits a complete hyperbolic metric of constant curvature $-1$. The Mostow-Prasad Rigidity Theorem then implies that to each such hyperbolic knot, one can associate a unique hyperbolic volume, in addition to a variety of other hyperbolic invariants. 

\bigskip

More generally, Thurston showed that a link in a compact 3-manifold is hyperbolic if and only if its complement contains no \emph{essential} disks, spheres, tori or annuli, where a surface is essential if it is incompressible and not parallel into the boundary.   

\bigskip

In this  paper, we extend the notion of hyperbolicity to the category of virtual links, introduced by Kauffman in 1999 (see \cite{kauffman}) as a generalization of classical knot theory. 

\bigskip

A \emph{virtual link} $L$ is an equivalence class of link diagrams, where in the diagrams, in addition to the traditional classical crossings, we allow a new type of crossing called a \emph{virtual crossing}, denoted by circling the crossing.  A traditional diagram with no virtual crossings is called a \emph{classical diagram}. Classical diagrams are considered a subset of the broader category of virtual diagrams.

\begin{figure}[h!]
    \centering
    \includegraphics[width=0.2\textwidth]{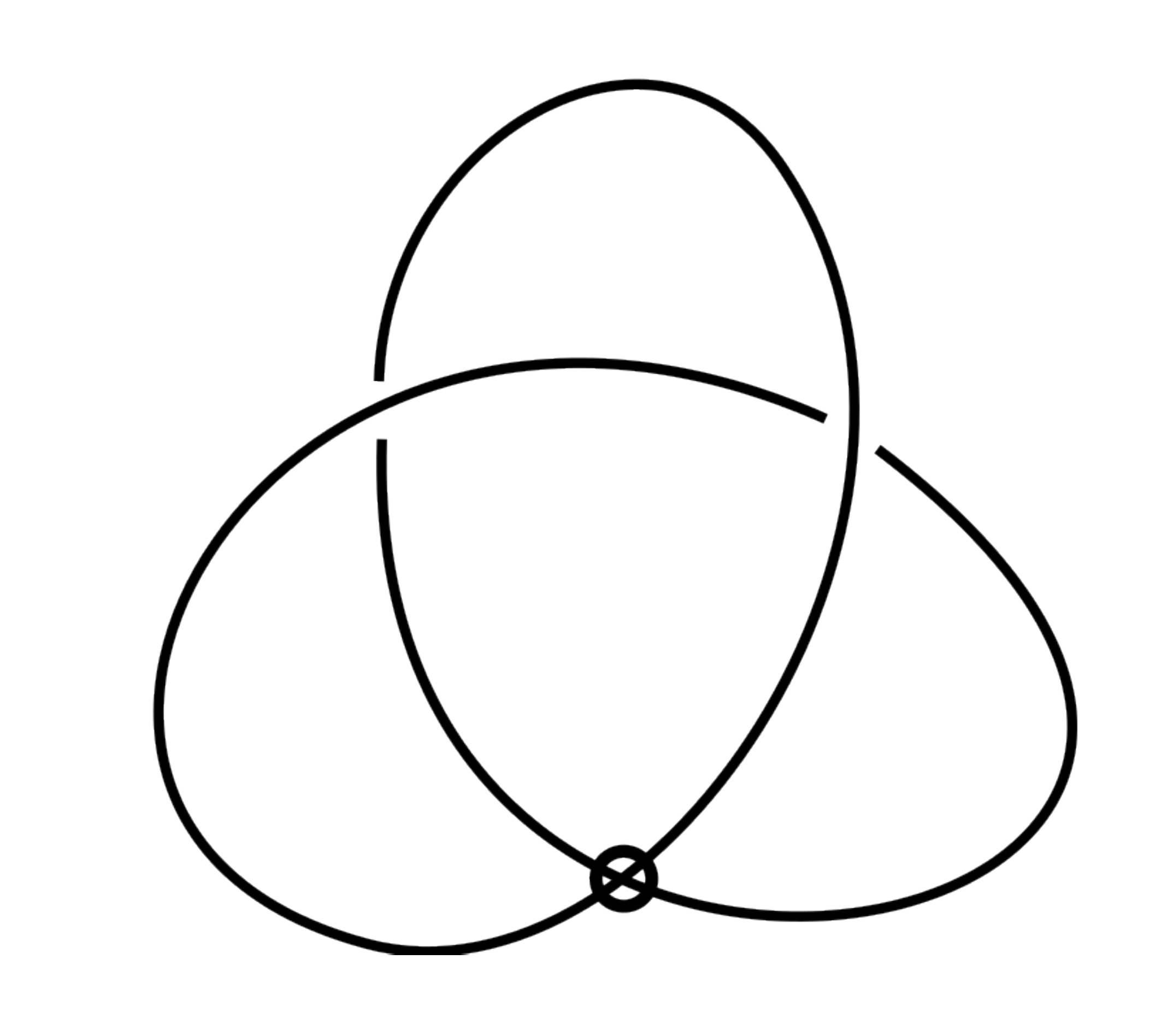}
    \caption{A diagram of the virtual trefoil knot.}
    \label{fig:virt_tref}
\end{figure}

Just as we have a set of Reidemeister moves that describes an equivalence on classical diagrams, there is a set of virtual Reidemeister moves that describes an equivalence on virtual diagrams. 

\bigskip

There is an equivalence between virtual links and links embedded in thickened surfaces, the projections of which will be link diagrams on a surface.

\begin{definition}
A \emph{surface-link pair}, $(S,L')$ is an equivalence class of classical link diagrams on an orientable surface $S$, considered up to classical Reidemeister moves on the surface. As in the classical case, these are in natural bijection with isotopy classes of tame embeddings of disjoint unions of circles  into $S \times I$.
\end{definition}

We will pass between a surface-link pair $(S,L')$ and the corresponding thickened surface and link $(S \times I, L')$ as needed, using $L'$ to represent both the projection to $S$ and the link in $S \times I$. 

\begin{definition}
A \emph{ribbon surface} $R(P)$ of a virtual link projection $P$ of a link $L$ is a surface-link pair constructed by the following algorithm. At each classical crossing of $P$, insert a disk, and connect the disks with flat untwisted band neighborhoods of the strands connecting the crossings of $P$, allowing an arbitrary choice of one of the two band neighborhoods to pass under the other at each virtual crossing of two strands. Once the ribbon surface is obtained, we cap off each boundary component with a disk. We say that the surface-link pair obtained in this manner \emph{represents} the original virtual link $L$.
\end{definition}

\begin{figure}[h!]
    \centering
    \includegraphics[width=0.2\textwidth]{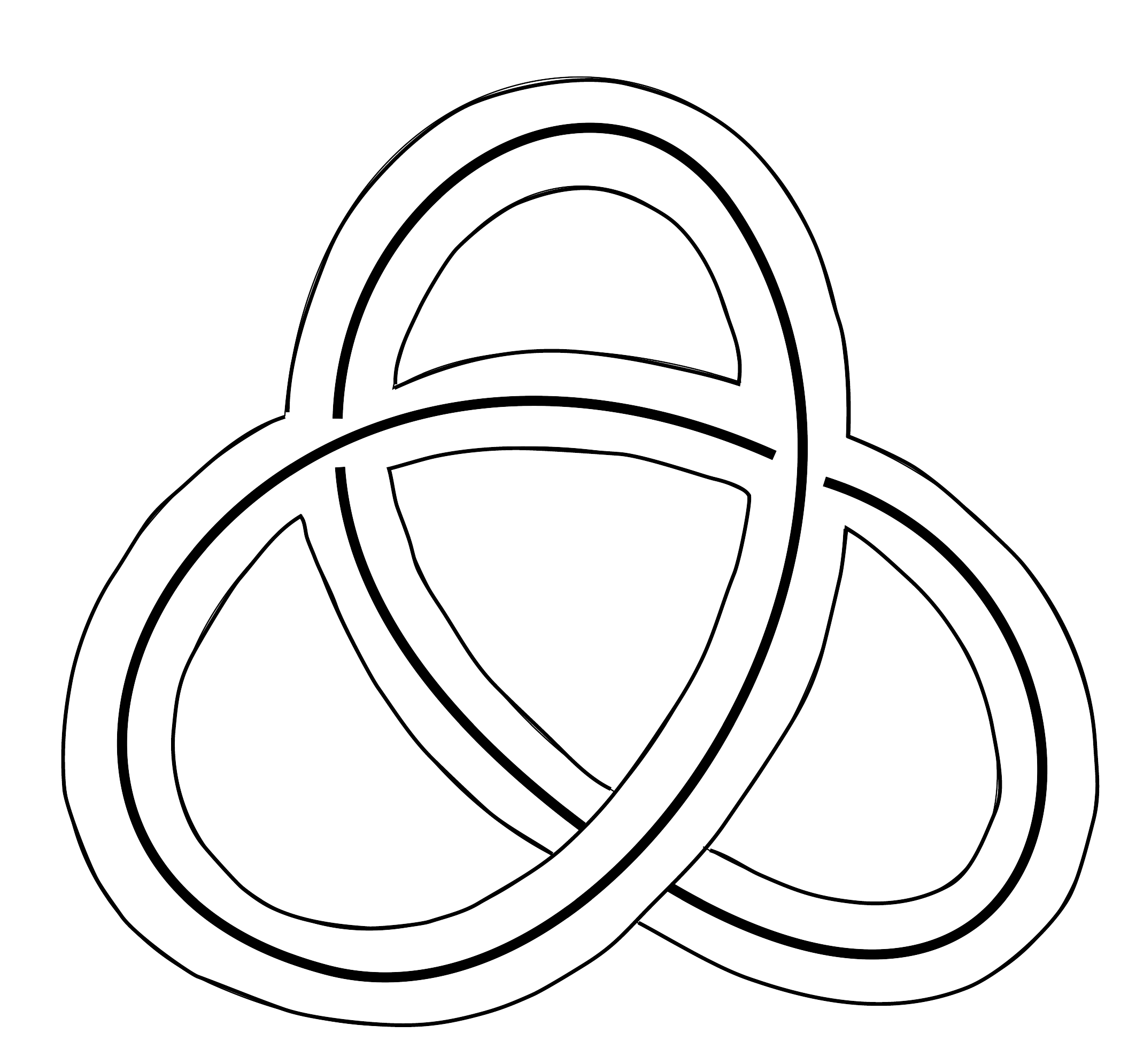}
    \caption{The ribbon surface associated to the above diagram of the virtual trefoil.}
    \label{fig:virt_tref_rib}
\end{figure}

In 2002, Carter, Kamada and Saito \cite{carter-kamada-saito} showed that the theory of virtual knots and links is equivalent to the theory of classical knots and links living in thickened closed, orientable, compact surfaces $S\times I$, allowing for stabilizations and destabilizations.  A \emph{stabilization} is the addition of a handle to $S$ avoiding $L$, and a \emph{destabilization} corresponds to cutting a surface-link pair $(S, L)$ along a simple closed curve $C$  on $S$ disjoint from $L$, and capping off the resulting boundaries with a disk. If a destabilization along $C$ reduces the genus of $S$, $C$ is called a \emph{cancellation curve}.

\bigskip

In \cite{kuperberg}, a fundamental theorem for the theory was proved. 

\begin{theorem}[Kuperberg, 2002]{\label{thm:kuperberg}}
Every stable equivalence class of links in thickened surfaces has a unique irreducible representative.
\end{theorem}

From this theorem, we know that we may assign to each virtual link $L$ a canonical surface-link pair $(S_g, L')$. Note that by Theorem \ref{thm:kuperberg} we may define $g$ to be the \emph{genus}, or \emph{minimal genus}, of the link $L$. It follows that $g$ is an invariant of virtual links. A surface-link pair $(S_g, L')$ is called \emph{minimal genus} if $g$ is the minimal genus of $L$. By Theorem \ref{thm:kuperberg} the minimal genus surface-link pair of $L$ is unique.

\bigskip

It is natural to ask whether certain virtual links are hyperbolic when they are realized as a link in a thickened surface. Indeed, we may apply Thurston's hyperbolicity criterion for $3$-manifolds to a minimal genus surface-link pair in order to obtain some notion of hyperbolicity. However, some additional criteria must be met by a virtual link in order to ensure the hyperbolic volume is well-defined, which is called \emph{tg-hyperbolicity}.

\begin{definition} A compact orientable 3-manifold $M$ is \emph{tg-hyperbolic} if, after capping off all sphere boundaries with balls, and discarding all torus boundaries, the resulting manifold $M'$ has a finite volume hyperbolic metric such that all remaining boundary components are totally geodesic. We also say $M'$ is a tg-hyperbolic manifold.
\end{definition}

We do not distinguish between the manifold obtained by removing a closed neighborhood of a link and the manifold obtained by just removing the link, since they are homeomorphic.

\begin{definition}
A surface-link pair $(S, L')$ is \emph{tg-hyperbolic} if the complement $(S \times I) \setminus L'$ is tg-hyperbolic.  A virtual link $L$ is said to be \emph{tg-hyperbolic} if there exists a tg-hyperbolic surface-link pair $(S, L')$ representing $L$. We will prove that such a pair must be the minimal genus representative and hence it is unique. We write $vol(L)$ for the tg-hyperbolic volume of $S \times I \setminus L'$.
\end{definition}

Let $N(L')$ be an open neighborhood of the embedding of $L'$ in $S \times I$. By work of \cite{thurston}, the existence of such a hyperbolic metric is equivalent to the fact that the manifold $(S \times I) \setminus N(L')$ possesses no properly embedded essential spheres, annuli, or tori. (Note that essential disks need not be considered since $\partial (S \times I)$ is always incompressible in $S \times I$ and if there were an essential disk with boundary in $\partial N(L')$, then there would also be an essential sphere.)

\bigskip

\noindent The additional criterion of having a totally geodesic boundary when $g \geq 2$ is to ensure that the link complement $(S_g \times I) \setminus L'$ has a unique finite volume. This condition is equivalent to having a hyperbolic metric on the manifold obtained by doubling across the boundary of $(S \times I) \setminus L'$ by Mostow-Prasad Rigidity. 


\begin{theorem}\label{thm:hyperbolic-minimal-genus}
If surface-link pair $(S, L')$ representing virtual link $L$ is tg-hyperbolic, then it is the unique minimal genus representative of $L$.
\end{theorem}

\begin{proof}
If $(S, L')$ is tg-hyperbolic, then it contains no essential spheres, disks, tori or annuli by work of Thurston \cite{thurston}. But as pointed out by Kuperberg (\cite{kuperberg}), the criteria for being minimal genus is the nonexistence of any essential spheres, disks or annuli with both boundaries in $\partial S \times I$. So $(S, L')$ is the unique minimal genus representative. 
\end{proof}

\noindent In Section \ref{alternating}, we consider virtual alternating links. 

\begin{definition} A \emph{virtual alternating link diagram} is a virtual link diagram in the plane such that the classical crossings alternate between over and under crossings as we travel around the diagram, ignoring virtual crossings.
\end{definition}

A theorem from \cite{SMALL2017} shows that a fully alternating link in a thickened surface,which is to say one that possesses an alternating projection such that all complementary faces are disks, must have a tg-hyperbolic complement. (See also related work in the genus one case in \cite{CKP} and in the arbitrary genus case in \cite{HP}.) We use this to prove that a prime nontrivial virtual alternating link that is not a classical 2-braid is tg-hyperbolic. Further, we explain how one determines primeness. As an example, we use the results from this section to prove tg-hyperbolicity for a particular class of fully alternating hyperbolic links, the \emph{virtual $n$-polygonal links}.

\bigskip

In Section \ref{kishino}, we consider certain classes of non-alternating virtual links. We prove that a virtual link obtained from a reduced prime connected alternating diagram by converting one crossing to virtual is tg-hyperbolic. Further, we consider a generalization of the Kishino knot. The Kishino knot has been of particular interest in virtual knot theory as an example of a nontrivial knot that is the connect sum of two trivial virtual knots. There are various proofs of its nontriviality. For instance,  in \cite{dye-kauffman}, it is proved that its virtual genus is 2 and it is therefore nontrivial.  Here, we generalize the Kishino construction to provide additional examples of non-alternating virtual links that are tg-hyperbolic. In order to prove this, we introduce an operation that appends a half-Kishino knot to any minimal genus generating projection of a tg-hyperbolic link, and results in a tg-hyperbolic link with genus one greater. We also use this to obtain tg-hyperbolic virtual links with the minimum possible number of crossings per genus. 

\bigskip
In Section \ref{hyperbolic-volume}, we determine some bounds for volumes of links embedded in thickened surfaces of various genera. In Section \ref{conjectures}, we list some conjectures on volumes of virtual links, and include what is known to support these conjectures.

\bigskip

\noindent Section \ref{table} presents a table of hyperbolic volumes of virtual knots of up to four classical crossings, obtained using the computer program SnapPy (\cite{SnapPy}), which was critical to these investigations.There are several interesting phenomena. First, of the 116 nontrivial virtual knots of four or fewer classical crossings, all but the classical trefoil knot are tg-hyperbolic. Second, there are examples of subsets of two to four of the knots that all have the same volume. It would be interesting to understand why this is the case. 

\bigskip

In a separate paper (c.f.\cite{small18}), we consider the Turaev surface construction for knots and links. In the case of classical alternating and virtual alternating link projections, this is the same construction as the ribbon surfaces described above. But for other link projections, the Turaev surface construction yields a surface-link pair that is always fully alternating and therefore, if prime, is tg-hyperbolic. This allows the extension of hyperbolic invariants to all classical and virtual link complements.



\section{Hyperbolicity of Virtual Alternating  Links}\label{alternating}

\begin{definition} Let $(S, L')$ be a surface-link pair. We say $(S,L')$ is \emph{fully alternating} if $L'$ is alternating and its complementary regions on $S$ are all open disks.
\end{definition}

\begin{definition} Let $(S, L')$ be a surface-link pair. We say $L'$ is prime if there does not exist a ball $B$ in $S \times I$ such that $\partial B$ intersects $L'$ transversely at two points and such that $B \cap L'$ is not an unknotted arc of $L'$.
\end{definition}

\noindent In \cite{SMALL2017}, two relevant theorems were proved.

\begin{theorem}\cite{SMALL2017}
\label{prime}A reduced fully alternating link $L'$ on a surface  $S$  is prime if and only if there does not exist a disk $D$ on $S$ such that $\partial D$ intersects $L'$ twice transversely and there are crossings in $D\cap L'$.
\end{theorem}

\noindent When there are no such disks, we say $(S,L')$ is \textit{obviously prime}. Hence, the theorem says that a reduced fully alternating link in a thickened surface is prime if and only if it is obviously prime. 

\begin{theorem}\cite{SMALL2017}
\label{tg}A prime fully alternating link in an orientable surface of genus at least one is tg-hyperbolic. 
\end{theorem}

From this, we can obtain the following theorem. 

\begin{theorem}
\label{virtualalternating}A nontrivial virtual alternating link $L$ that is not a classical 2-braid and that has a prime representation as a surface-link pair coming from a virtual alternating projection is tg-hyperbolic. 
\end{theorem}

\begin{proof} In \cite{menasco}, Menasco proved that for a classical link, a nontrivial prime alternating link in $S^3$ that is not a 2-braid link is hyperbolic. A reduced alternating projection of such a link yields a surface-link pair of genus 0. Thus, in this case, the result holds. For a non-classical virtual alternating link $L$, the surface-link pair $(S, L')$ representing a virtual alternating projection of $L$ must be fully alternating with genus at least one. By Theorem \ref{tg}, it is tg-hyperbolic.  
\end{proof}

\noindent In order for Theorem \ref{virtualalternating} to be useful, we need to know how to determine whether a given virtual alternating link is prime. Let $L$ be such a link and let $(S, L')$ be the surface-link pair representing a virtual alternating projection $P$. Suppose that $L'$ is composite. Then by Theorem \ref{prime}, there exists a disk $D$ on $S$ such that $\partial D$ intersects $L'$ twice and there are crossings in $D\cap L$ after $D \cap L$ is reduced.(Note that a result of Menasco (\cite{menasco}) implies that the 1-tangle $T$ inside $D$ is nontrivial.)
\\\\
But then we can isotope the link $L'$ on $S$ so that the disk appears on the top side of the surface and there are no strands of the knot directly beneath the disk. In other words, if we project vertically down, we obtain a projection of the original virtual link such that there is a disk in the plane containing the projection of the 1-tangle, and there are no virtual or classical crossings of this 1-tangle with the rest of the projection. Hence this projection has a reduced alternating 1-tangle that makes it apparent that $(S, L')$ will not be prime. 
\\\\


\noindent Be that as it may, finding such a projection may be difficult and proving no such exists even more difficult. Hence, we turn to Gauss codes.
\\\\
\noindent Gauss codes were the original motivation for investigation into virtual knot diagrams. See \cite{kauffman}, for instance. 
A Gauss code is a sequence of integers from the set $\{1, \dots, n\}$, each integer corresponding to one of the $n$ crossings, and each appearing twice, each followed by one of the letters $O$ and $U$ (for ``over'' and ``under''). The integers appear in the order one encounters crossings as one travels around the projection, and the letter following each number corresponds to whether one is passing through the crossing on an overpass or an underpass. 
We will always consider Gauss codes cyclically, so there is no particular starting point.

\begin{definition}
A Gauss code is \textit{reduced} if the two appearances of a given integer are never adjacent. 
\end{definition}

\noindent Note that when a Gauss code is not reduced, one can always use generalized Reidemeister moves on the corresponding projection to reduce it, which just eliminates the given integer.

\begin{definition}A \emph{subcode} of a Gauss code is a sequential proper subword that includes both appearances of each integer for all of the integers appearing within it. A \emph{classical subcode} is a subcode that can be realized with no virtual crossings.
\end{definition}

\begin{theorem}
\label{code} A virtual alternating knot $K$ is prime if and only if the Gauss code corresponding to any reduced alternating virtual projection contains no classical alternating subcodes.
\end{theorem}

\begin{proof}It is useful to note that the Gauss code obtained for a virtual link projection $P$ of a virtual link $L$ is the same as the Gauss code obtained for the link $L'$ on the surface $S$ coming from the surface-link pair $(S, L')$ representing $P$. Hence, the 1-tangle $T$ appearing in the disk $D$ on $S$ generates a subword of the Gauss code that is classical. Thus, the same is true for the identical Gauss code of the projection $P$.
\\\\
\noindent Let $P$ be any reduced virtual alternating projection of a virtual alternating link $L$. Let $(S, L')$ be the surface-link pair that represents it. Then $L'$ is fully alternating as a link in $S \times I$.
\\\\
\noindent The Gauss code of $P$ contains a classical alternating subcode if and only if the Gauss code of $L'$ on $S$ contains such a code. This occurs if and only if $L'$ is obviously composite on $S$, and therefore if and only if $S \times I \setminus L'$ is composite. 
\\\\
\noindent It only remains to prove that if one reduced virtual alternating projection $P$ of $L$ does not contain a classical alternating subcode, then the same is true for any other. 
In \cite{SMALL2017} (see Lemmas 8 and 9), it was proved that such a surface-link pair can have no cancellation curves. Hence, $S \times I \setminus L'$ must be an irreducible representative of $L$. By Theorem \ref{thm:kuperberg}, it must be the unique such. Hence, if it is prime, so must be all other surface-link pairs representing reduced virtual alternating projections of $L'$.
By Theorem \ref{prime}, it  then follows that if one reduced virtual alternating projection $P$ of $L$ does not contain a classical alternating subcode, then the same is true for any other. 
\end{proof}

\noindent We note that if we have a virtually alternating link of more than one component, with a connected projection,  then we are looking for a subcode of the Gauss code of one component, together possibly with the complete Gauss codes of additional components such that every integer appearing in this set appears twice and together they generate an alternating classical 1-tangle.
\\\\

\subsection*{The $n$-Polygonal Links}

\noindent The rest of this section is devoted to an extended example that demonstrates the  utility of the various theorems in this section. 
\\\\
\noindent We define a \emph{$n$-polygonal virtual link} $\mathcal{P}_n$ to be an alternating virtual link with $n$ virtual crossings to the inside and $n$ classical crossings to the outside as in Figure \ref{fig:polygonal}. Note that if $n$ is divisible by $3$, $\mathcal{P}_n$ is a $3$-component link and otherwise it is a knot.

\begin{theorem}\label{thm:polygonal-genus} For $n \geq 3$, all $n$-polygonal virtual links are tg-hyperbolic and their minimal genus is $g(\mathcal{P}_n) = \lceil\frac{n}{2}\rceil-1$.
\end{theorem}

\begin{proof}
    Let $L$ be the $n$-polygonal virtual link in the standard projection. We construct the ribbon surface corresponding to it. The genus of this surface is $g = \frac{n - b + 2}{2}$ where $n$ is the number of crossings and $b$ is the number of boundaries. 
 \\\\
 \noindent In Figure \ref{fig:polygonal}, we see the four boundary components for the ribbon surface of the  6-polygonal link. In general, there will always be one outermost component as in blue, one intermediate component as in red, and, as also appears here in blue,  two innermost components for $n$ even and one innermost component for $n$ odd.Then $g = \frac{n-4+2}{2} = \frac{n-2}{2} = \frac{n}{2}-1 = \ceil{\frac{n}{2}}-1$ in the even case and $g = \frac{n-3+2}{2} = \frac{n+1}{2}-1 = \ceil{\frac{n}{2}}-1$ in the odd case.

    \begin{figure}
        \centering
        \includegraphics[width=0.6\textwidth]{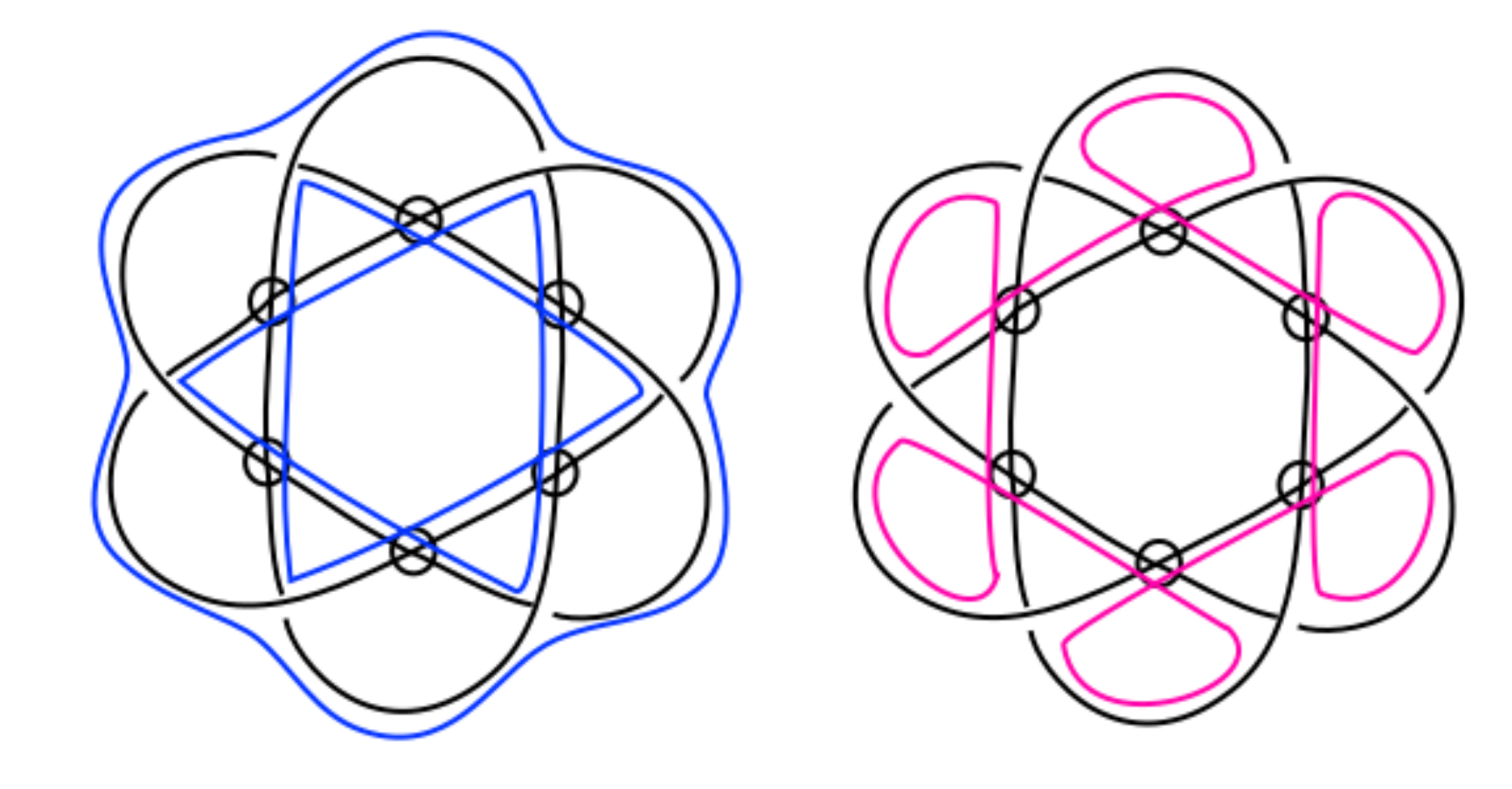}
        \caption{
        The boundaries of the ribbon surface, 3 for $n$ odd and 4 for $n$ even.}
        \label{fig:polygonal}
    \end{figure}
    
    \noindent Since $K$ is virtually alternating, the surface-link pair representing it is fully alternating. As mentioned previously, this implies by \cite{SMALL2017} that this is a minimal genus representative.
    \\\\
    \noindent To see that the complement is tg-hyperbolic, by Theorem \ref{tg}, it suffices to prove that the polygonal links are prime. By Theorem \ref{code}, we know this to be the case if and only if the corresponding Gauss code contains no classical alternating subcode.We argue that here, by showing there are no subcodes, classical or not. 
    \\\\
    \noindent Note that if there is a subcode, its complement is also a subcode.
\\\\
    \noindent In the case of a knot, so $n$ is not divisible by 3, the Gauss code of such a link can be obtained in the following manner. Write the integers $1,\dots, n, 1\dots n, 1,\dots, n$. Then cross off every third integer. By the time we are done, we will have crossed out one copy of each integer, leaving the two copies necessary for this to be a valid Gauss code. We will not add in the letters O and U since they play no role here.
\\\\
    \noindent The Gauss code naturally falls into three thirds, corresponding to each copy of the original $1,\dots,n$. Call these $X_1,X_2,$ and $X_3$. Each integer will appear twice, once each in two of the three thirds. For integer $a$, it appears in a different pair of thirds than either $a+1$ or $a+2$. The exact pair of thirds it falls into is determined by whether $n \equiv 1 (\mod{3})$ or $n \equiv 2 (\mod {3})$.
    \\\\
    \noindent Without loss of generality, suppose the integer 1 is in a subcode. Then there are two options. All entries between the two copies of 1 must be in the subcode or all entries outside the two copies of 1 are in the subcode. In particular either two of the thirds are in the subcode or one of the thirds is in the subcode. If two of the thirds are in the subcode, then every integer appears in the subcode, in which case the subcode contains all of the code, a contradiction. If one of the thirds is in the subcode, then 2 or 3 must be in the subcode, and they each fall into different pairs of thirds than 1. So the subcode again must contain at least two of the thirds, which means it must be the entire code, a contradiction.
    \\\\
    \noindent When $n $ is divisible by 3, so $L$ is a 3-component link, the Gauss codes of the three components are each obtained from $1, \dots, n$ by crossing off every third entry, starting either with the first, second or third entry. Call these three codes thirds for convenience.  Then a subcode must be a word in one of the three codes together with either both remaining thirds or one additional third or no additional thirds. Note that the complementary set must also be a subcode. Since as before, if a subcode contains two thirds, it must be all three codes, and therefore not a subcode at all, the only possibility is that the subcode contains one component's full code and a subword of another. The same then holds for the complementary code. But then consider the one component's full code in the subcode. It contains integers $a$ and $a+1$, which have their second copies occurring in different remaining thirds. So the subcode would have to contain entries from all three thirds, a contradiction.
    
\end{proof}

\section{Hyperbolicity of Certain Non-Alternating Virtual Links}\label{kishino}

In this section, we give proofs that certain non-alternating families of virtual knots are tg-hyperbolic, including the family of knots obtained by taking a reduced prime alternating projection and converting one crossing to virtual and the family of so-called generalized n-Kishino knots. In the process, we introduce a certain "composition" of a half-Kishino knot with any tg-hyperbolic knot that always results in another tg-hyperbolic knot with genus increased by 1.

\subsection{The 1-virtual alternating links}

\begin{definition} Let $P$ be any reduced connected alternating classical  projection such that there does not exist a  circle in the projection plane intersecting the projection twice with crossings to both sides. A {\it 1-virtual alternating link} is a link with projection obtained by converting any single crossing of $P$ into a virtual crossing.
\end{definition}

 Note that such an $L$ is not a virtual alternating link. 

\begin{theorem}\label{onevirtual} A 1-virtual alternating link $L$ is tg-hyperbolic.
\end{theorem}

\begin{proof}
The surface $S$ of the surface link pair $(S, L')$ corresponding to the given projection is a torus. The link $L'$ sits on the torus as in Figure \ref{fig:onevirtualtorus}, where $T$ is an alternating tangle.

\begin{figure}[htbp]
    \centering
    \includegraphics[scale=.4]{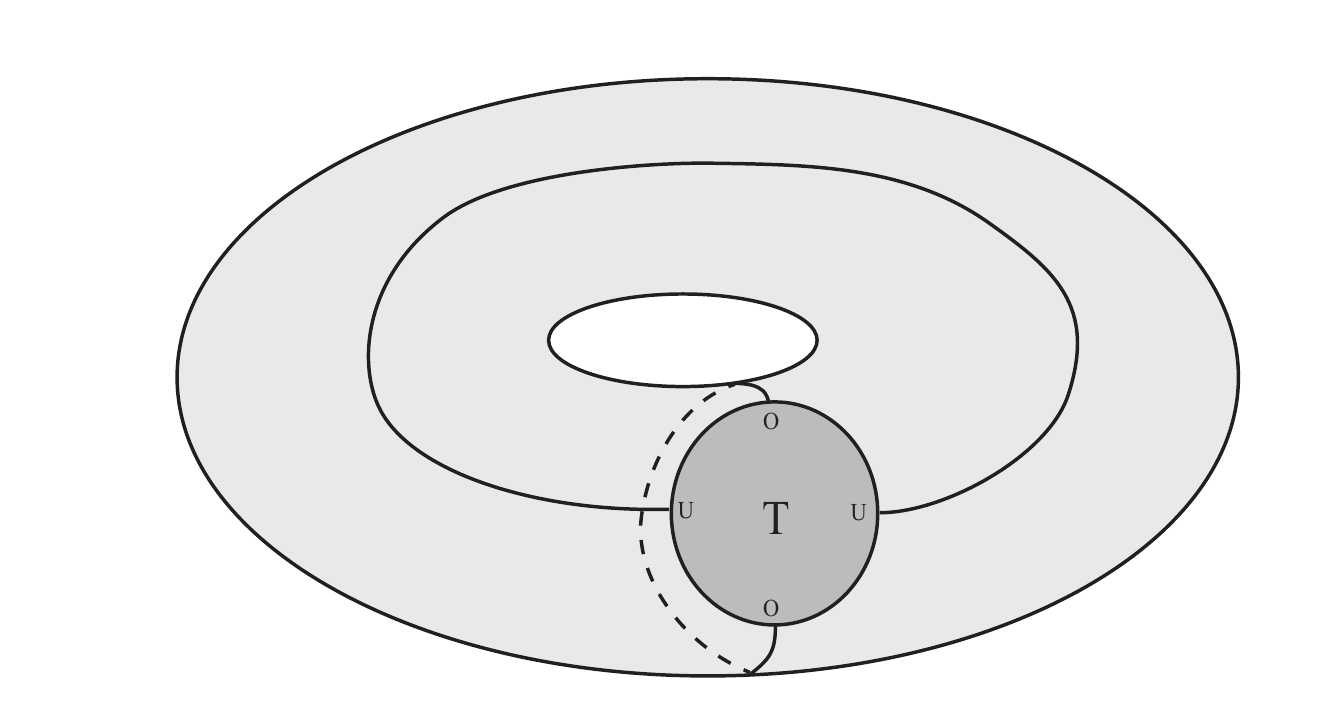}
    \caption{The surface-link pair for a 1-virtual alternating projection.}
    \label{fig:onevirtualtorus}
\end{figure}

Then $S \times I \setminus L'$ can be realized as a link complement in the 3-sphere by adding a Hopf link appropriately to the link $L'$.

\begin{figure}[htbp]
    \centering
    \includegraphics[scale=.6]{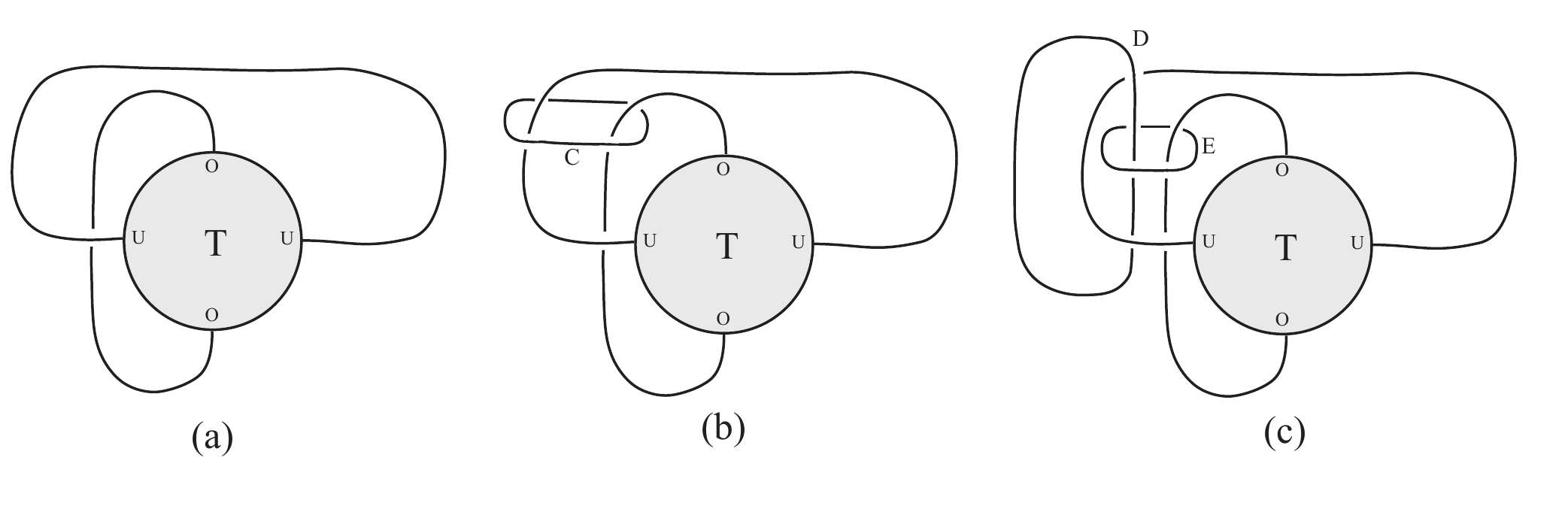}
    \caption{A 1-virtual prime alternating link is tg-hyperbolic.}
    \label{fig:onevirtual}
\end{figure}

\bigskip 

Projecting $L'$ to the plane yields the alternating projection in Figure \ref{fig:onevirtual} (a). Adding a single trivial component as in Figure \ref{fig:onevirtual} (b) results in a so-called augmented alternating link and preserves hyperbolicity by Theorem 4.1 of \cite{Adams2}. Replacing that component by two components as in Figure \ref{fig:onevirtual} (c) also preserves hyperbolicity (which is equivalent to tg-hyperbolicity when $S$ is a torus) by the Chain Lemma of \cite{AMR}.

\end{proof}

Note that this also implies that 1-virtual alternating links are nontrivial and non-classical,  and that they are all genus 1, since tg-hyperbolicty of the torus-link pair implies the genus cannot be less than 1. (See also \cite{dye-kauffman} for when 1-virtual links must be non-classical.)

\medskip
\subsection{The generalized n-Kishino knots}

We now consider a second family of non-alternating virtual knots called  the generalized $n$-Kishino knots. These are virtual knots that appear as in Figure \ref{fig:n_kishino}.  There are $n$ virtual crossings and $2n$ classical crossings, each classical crossing of which can go either way. But note that for any choice of crossings, none of these knots can be made virtually alternating. Thus, the previous theorems do not apply. Note that the generalized 2-Kishino knot with a particular choice of crossings is the original Kishino knot.

\begin{figure}[htbp]
    \centering
    \includegraphics[scale=.3]{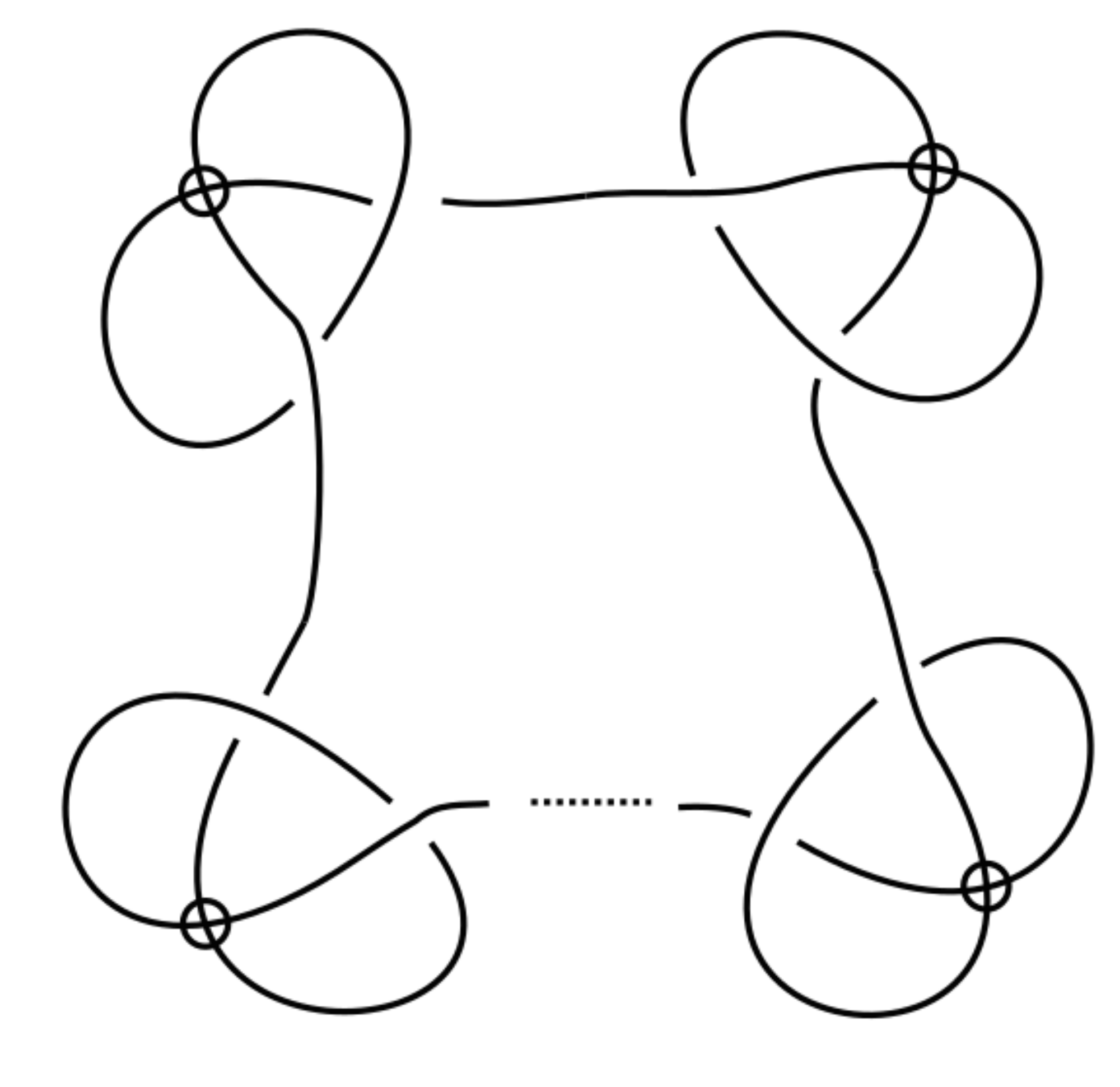}
    \caption{A diagram of a generalized $n$-Kishino knot with a particular choice of  classical crossings.}
    \label{fig:n_kishino}
\end{figure}

\begin{theorem}\label{thm:kishino-hyperbolic}Every generalized $n$-Kishino knot $K_n$ is tg-hyperbolic with genus $n$. 
\end{theorem}

To prove this, we first prove a general fact that will also be useful in other contexts. We then prove that the genus two generalized Kishino knot is tg-hyperbolic. And then we prove that for any tg-hyperbolic link, appending a half-Kishno knot preserves tg-hyperbolicity and increases genus by one. 

\begin{lemma}\label{lma: sphere-tori}
If $(S,K)$ is a surface-knot pair of minimal genus with $g(S) \geq 2$, and $(S\times I) \setminus K$ does not admit any essential tori, then $(S,K)$ is tg-hyperbolic.
\end{lemma}

\begin{proof}
We first note that $K$ is prime, since if not, we could create a swallow-follow torus that would contradict the non-existence of essential tori. 
\\\\
We also note that because $(S,K)$ is minimal genus, $S \times I \setminus K$ cannot admit any essential annuli with both boundary components on $\partial (S \times I)$, one on $S \times \{0\}$ and one on $S \times \{1\}$ as such an annulus would allow us to reduce genus. Such an annulus is called a {\it vertical annulus} and the curve on $S$ corresponding to it is what we previously called a {\it cancellation curve}.
\\\\\
We also note that there can be no essential spheres. If such a sphere existed, then in $S \times I$, it would have to bound a ball. So for the sphere to be essential, the knot would have to be to the ball side, meaning that there were numerous cancellation curves and the surface-link pair could not be minimal genus. 
\\\\
Now we show that $S \times I \setminus K$ cannot admit an essential annulus $A$ with both boundary components on the \textit{same} (without loss of generality $S\times \{0\}$) boundary of $S \times I$. Note that both boundary components of such an annulus are isotopic in $S \times I$ and hence are parallel on $S$.  Then, $A$ separates $S \times I$ into two components $C_1$ and $C_2$, so that $C_1$ is homeomorphic to $S \times I$, and $C_2$ is homeomorphic to a solid torus, such that $\partial A$ appears as two $(p,1)$-curves on $C_2$ for some $p \in \N$. Because $A$ is essential, $K$ must be completely contained in $C_2$. Thus, we can find a projection of $K$ that stays in a neighborhood of one of the boundaries of $A$. Then $(S, K)$ admits a cancellation curve, contradicting the genus-minimality of $(S,K)$.
\\\\
Thus, we can restrict ourselves to the case where a candidate essential annulus $A$ in $(S \times I) \setminus \Int{(N(K))}$ has: (1) boundary strictly on $\partial N(K)$ or (2) boundary on both $\partial S \times I$ and $\partial N(K)$. 
\\\\
In the first case, by the same argument as in Lemma 10 of \cite{SMALL2017}, we conclude that $g(S) = 1$, a contradiction.

Case (2) is eliminated by applying the argument from Lemma 11 of \cite{SMALL2017} to again show the genus is 1, a contradiction. 
\\
Thus, $(S \times I) \setminus K$ does not admit any essential spheres/tori/annuli, so Thurston's Hyperbolization Theorem gives the result.
\end{proof}

Next, we prove the genus two case.

\begin{lemma}\label{Kishino2}The genus two generalized Kishino knot $K$ is tg-hyperbolic.
\end{lemma}

\begin{proof} That there exists no cancellation curve for the standard Kishino knot is proved in \cite{dye-kauffman}, using the surface bracket polynomial. That same proof easily also proves that the same is true when the four crossings are switched to the other possibilities since the set of states is unchanged, and changes to the coefficients still allow for the proof to go through.
\\\\
Consider the minimal genus representation, as in Figure \ref{fig:Kishinogenus2}, given by $(S, K')$. 

\begin{figure}[htbp]
    \centering
    \includegraphics[scale=.7]{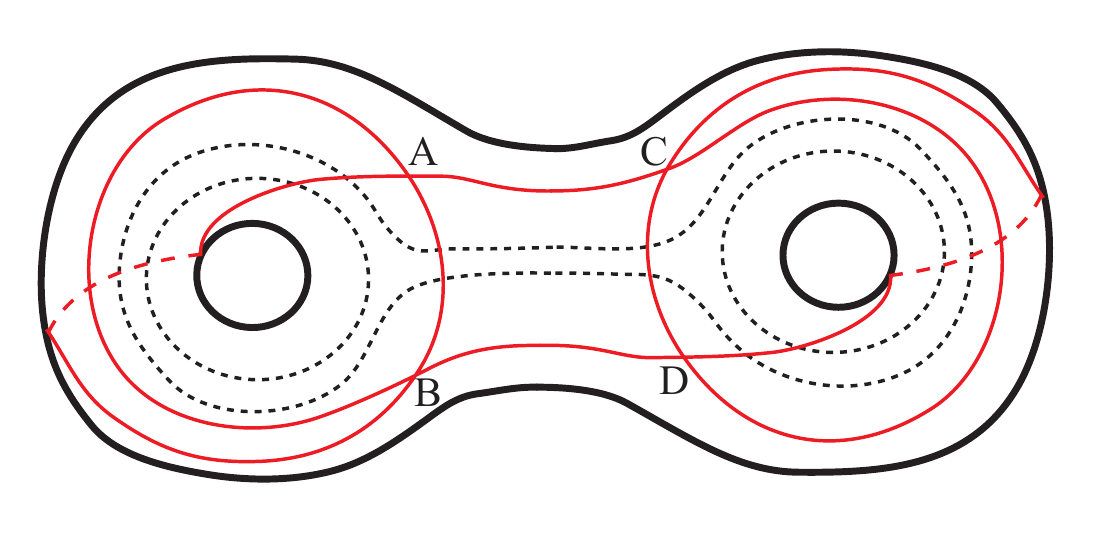}
    \caption{A representation of the generalized genus two Kishino knot where crossings $A, B, C$ and $D$ can be chosen arbitrarily.}
    \label{fig:Kishinogenus2}
\end{figure}

Let $S \times I$ be the thickened genus two surface, with $S \times \{0\}$ being the inner boundary. Let $M$ denote $(S\times I)\setminus N^\circ(K')$. Consider the twice-punctured annuli $Q_1$ and $Q_2$ of  such that $\partial Q_i \cap (S\times\{0\})$ and $\partial Q_i \cap (S\times\{1\})$ are meridians of the $i$-th handle of $S\times\{0\}$ and $S\times \{1\}$ respectively. Let $W = M \setminus (N^\circ(Q_1) \cup N^\circ(Q_2))$ be the resulting handlebody.  Let $Q_i'$ and $Q_i''$ be the two copies of $Q _i$ that result on $\partial W$. In Figure~\ref{fig:cut_kishino}, we depict $W$ (although the neighborhoods of $N(K')$ have not be removed in this picture.). 


\begin{figure}[htbp]
    \centering
    \includegraphics[scale=.3]{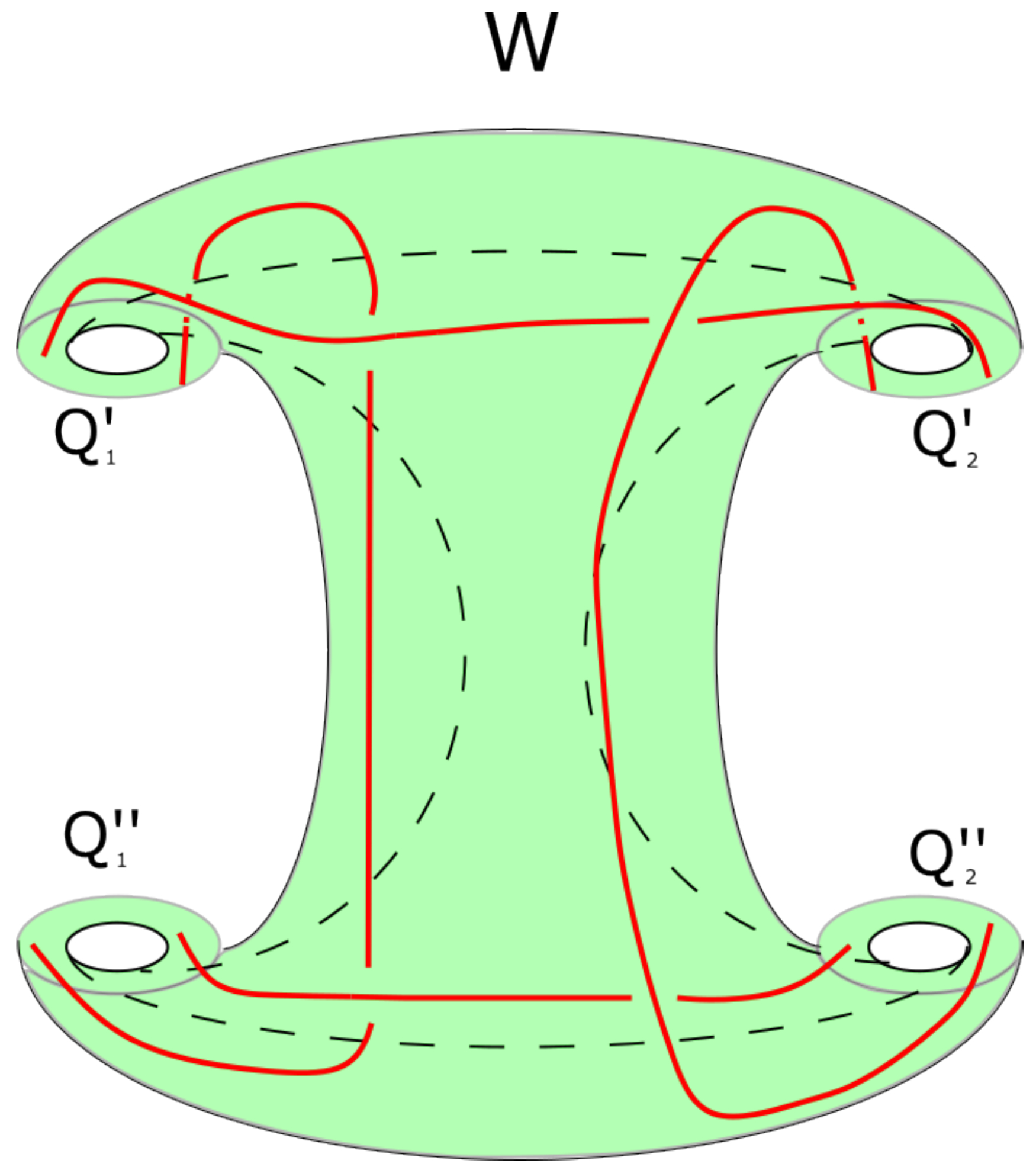}
    \caption{The handlebody $W$ for $K'$}
    \label{fig:cut_kishino}
\end{figure}

$V^* = V \bigcap W$.
\\\\


We prove that there are no essential tori in $M$. Suppose that $T$ were such an essential torus, and let $T^* = T \bigcap W$. 
\\\\
First, consider the case that $T^* \cap Q = \emptyset$. Then $T^* = T$ must separate $W$ into two components $W \setminus T^* = \mathcal{T}_1 \cup \mathcal{T}_2$, where $\partial W \subset \mathcal{T}_1$. Since $\partial W$ is connected, $\mathcal{T}_2$ is either a solid torus or a nontrivial knot exterior, and solid torus is not a possibility by the incompressibilty of $T$.  Further,  since $T^*$ is incompressible, $\pi_1(\mathcal{T}_2)$ injects into $\pi_1(W)$. Since $W$ is a handlebody, $\pi_1(W)$ is free, as are all its subgroups. However, $\pi_1(\mathcal{T}_2)$ is not free, creating a contradiction. Therefore, $T^* \cap Q$ is nonempty. 
\\\\
\noindent Take $T$ such that $T^* \cap Q$ has a minimal number of intersection curves.  Without loss of generality, take some  $\gamma \subset T^* \cap Q_i'$, trivial on $T^*$.
\\\\
If $\gamma$ is trivial on $Q_i'$, then we can isotope the disk $D$ it bounds on $T^*$ to $Q_i'$, contradicting the minimality of the number of intersections of $T$ with $Q$. If $\gamma$ is nontrivial on $Q_i'$, $\gamma$ wraps around one, two or three handles of $W$, each one time. Such a simple closed curve is not a trivial element in the free group $\pi_1(W)$, since it is the product of one, two  or three generators, but it bounds the disk $D \subset W$, a contradiction.
\\\\
\noindent If $\gamma$ is a nontrivial curve on $T$, then all components of $T^* \cap Q$ are parallel curves in $T$ and all must be nontrivial on $Q_i'$  and $Q_i''$ since $T$ is incompressible. 
Let $\partial_0Q_i' = Q_i' \cap S \times \{0\}$,  
$\partial_1Q_i' = Q_i' \cap S \times \{0\}$, and $n_1$ and $n_2$ be the boundaries of $N(L) \cap Q_i'$. Then, $\partial Q_i' = n_1 \cup n_2 \cup \partial_1Q_i' \cup \partial_0Q_i'$ and $Q_i' \setminus \gamma = \mathcal{Q}_1 \cup \mathcal{Q}_2$ where $\partial_1Q_i' \subset \mathcal{Q}_1$.
\\\\
All components of $T^*$ are annuli that separate $W$ into two pieces. Suppose there is such an annulus $A$ with both boundary-components on the same $Q'_i$. Then we argue that the pair of boundary-components must also bound an annulus $A'$ on $Q'_i$. Otherwise,  since $\partial N(K') \cap W$ is a set of four annuli, each of which avoids $A$ and each of which begins and ends on a different copy of $Q'_i$ or $Q''_i$, $A$ would have to intersect $\partial N(K')$ or $A$ would have to intersect the inner boundary $S \times \{ 0\} \cap \partial W$, which it cannot.
\\\\
\noindent Thus, $A \cup A'$ is a torus $T'$ which is unknotted and has longitude $\gamma$ on $Q_i'$. In particular, it bounds an unknotted solid torus, through which $A$ can be isotoped to $A'$ and then a bit further to eliminate two intersection curves, contradicting the choice of $T$ as a torus with a minimal number of intersection curves.
\\\\
\noindent Thus, the two boundaries of $A$ must occur on distinct $Q_i'$ and/or $Q_i''$. However, since the inner boundary $S \times \{ 0\} \cap \partial W$ intersects all $Q'_i$ and $Q''_i$, it must be the case that $\gamma$ does not separate $\partial_0Q_i$ from $\partial_1Q_i$.
\\\\
\noindent Therefore, the only possibilities left are that $\gamma$ encircles one or both of $n_1$ and $n_2$. It cannot encircle both, as the annuli in $\partial N(K_n')$ that begin at them end at different copies of the components of $Q$. If $\gamma$ encircles one, then the annulus $A''$ in $\partial N(K_n')$ that begins at it must end at another component of $Q$, for convenience say $Q_j'$. Then to avoid intersecting that strand of $K_n'$ inside $A''$, $A$ must also end on $Q_j'$, encircling the other boundary of $A''$. Then $A$ must be parallel to $A''$.
\\\\
Thus, the only possible components in $T^*$ are annuli, each of which is parallel to an annulus in $\partial N(K_n')$.These glue together to form a $T$ that is boundary-parallel, a contradiction. 
\\\\
By Lemma \ref{lma: sphere-tori}, $K$ is then tg-hyperbolic.
\end{proof}

As an aside, the technique applied above to demonstrate the hyperbolicity of the generalized $2$-Kishino knots can be used to show that a larger class of virtual knots are tg-hyperbolic. Let $K$ denote a virtual knot, and let $(S_n,K)$ denote its minimal genus representation, with $n = g(S) \geq 2$. Let $M = S_n \times I \setminus N(K)$. Consider the collection $Q_1, \ldots, Q_n$ of twice-punctured annuli such that $\partial Q_i \cap (S_n\times\{0\})$ and $\partial Q_i \cap (S_n\times\{1\})$ are meridians of the $i$-th handle of $S\times\{0\}$ and $S\times \{1\}$ respectively. Let $W$ be the closure of $M\setminus\left(\bigcup_{i=1}^n N^\circ(Q_i)\right)$. As in Figure \ref{fig:cut_kishino}, forming $W$ by removing a regular neighborhood of each $Q_i$ leaves us with boundary annuli $Q_i', Q_i''$.

\begin{theorem}
 Let $K$ be a virtual knot such that for its minimal genus representation, $W$ is a handlebody, and  any pair of strands of $K$ intersecting some fixed $Q_i'$ or $Q_j''$ lead to different $Q_i',Q_j''$ in $W$. Then $K$ is tg-hyperbolic.
\end{theorem}

Techniques of \cite{dye-kauffman} can be used to show that a given representation is minimal genus.

\bigskip

We now introduce an operation called {\it appending a half-Kishino knot} that one can perform on a virtual link projection to obtain a new virtual link.  Given any virtual link projection $P$, we compose it with a projection as shown in Figure \ref{fig:kishinoappend}. Note that the classical crossings labelled $A$ and $B$ can each be chosen either way. For the knot that is being used on the right side in the composition, two of the crossing choices yield the virtual trefoil and the other two yield the trivial knot. However, the composition must be at the point shown on the arc of the knot. We call this {\it appending a half-Kishino knot}, since the knot that we are appending is half of a generalized Kishino knot.  Note that the arc on the projection $P$ to which the half-Kishino knot is being appended can be anywhere in the projection, not just on an exterior strand, and we can append the half-Kishino or its reflection. But no new crossings can be created in the process. We see the impact on the corresponding surface-link pair in the figure. 

\begin{figure}[htbp]
    \centering
    \includegraphics[scale=.7]{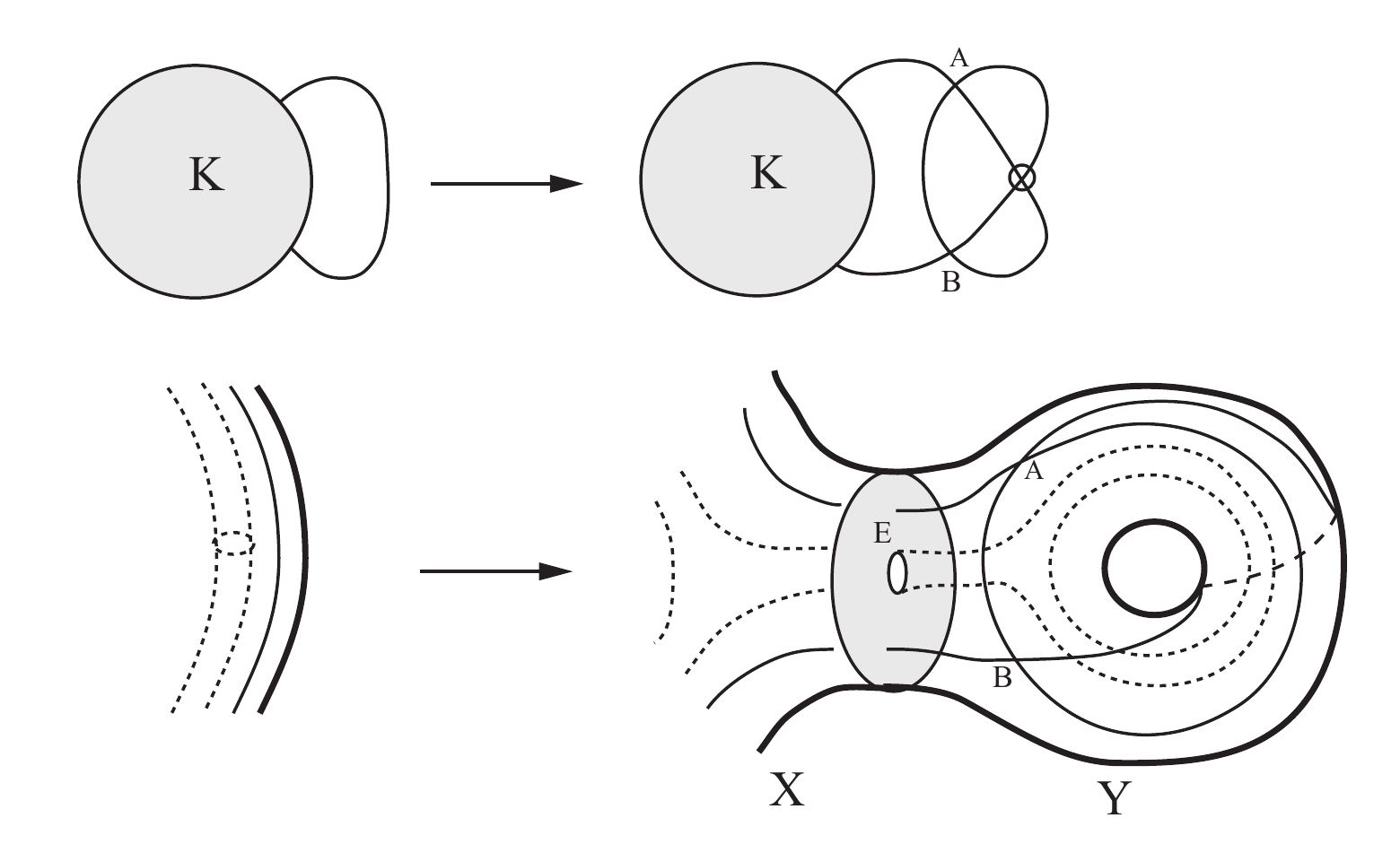}
    \caption{Appending a half-Kishino to a virtual knot.}
    \label{fig:kishinoappend}
\end{figure}

\begin{lemma} \label{kishinoappend} Given a tg-hyperbolic virtual knot $K$ of genus at least 1, appending a half-Kishino knot to a projection that corresponds to a minimal genus representation of $K$ yields a tg-hyperbolic knot $J$ with $g(J) = g(K) +1$. \end{lemma}

\begin{proof} Starting from a minimal genus representation $(S, K')$ for $K$, we can obtain a surface-link pair $(S', J')$ for $J$ as appears in Figure \ref{fig:kishinoappend}. Let $R = S' \times I$. The annulus $E$ splits $R$ into two submanifolds, $X$ and $Y$ where $Y$ is the side containing the half-Kishino append. 
\\\\
We first prove that there are no disks $D$ in $R \setminus J'$ such that $D \cap E = \partial D$ and $\partial D$ is nontrivial in $E \setminus (J' \cap E)$. Let $\gamma =  \partial D$. 
Note that $\gamma$ cannot separate the boundaries of $E$, as then those boundaries would be trivial in $R$. But $R$ is a thickened surface and as such those curves are nontrivial.
\\\\
If $\gamma$ contains one puncture on $E$, we could form a once-punctured sphere in $R$, a contradiction. If $\gamma$ contains both punctures,we could isotope $D$ into $E$. If $D$ is on the $Y$ side of $R$, then in the process, we would  push $Y \cap J'$ into a neighborhood of $E$, contradicting the fact $Y \cap J'$ is a $(1,2)$-curve in the thickened torus that is $Y$. If $D$ is on the $X$ side of $R$, then in the process, we would  push $X \cap J'$ into a neighborhood of $E$, providing room for cancellation curves and contradicting the fact the original representation of $K$ was minimal genus.
\\\\
We next show that $S'$ is minimal genus by showing there are no vertical annuli in $R \setminus J'$ (those with boundaries in $S \times \{0\}$ and $S \times \{1\}$). 
\\\\
Suppose a vertical annulus $A$  exists. We first show that it intersects the annulus $E$. If not, it must either be entirely in $X$ or $Y$. It cannot be in $X$ because $(S, L')$ is minimal genus. If it is in $Y$, we could double $Y$ across $E$ to obtain a genus two generalized Kishino knot which would still contain a vertical annulus, a contradiction to Lemma \ref{Kishino2}.  
\\\\
So $A$ must intersect $E$. Assume it has been isotoped to minimize the number of intersection curves.  Then, because of the fact we have eliminated disks with nontrivial boundaries in $E$,  all intersection curves are arcs. If there exists such an arc with both endpoints on the same boundary component of $E$, then that arc must cut a disk $D$ from $A$ and a disk $D'$ from $E$. Then $D'' = D \cup D'$ is a disk with trivial boundary in either $S \times \{0\}$ or $S \times \{1\}$. Then we can either isotope away the intersection or if $J'$ intersects $D'$, it must do so twice, and again we can isotope the knot to remove the two intersections with $E$, a contradiction.
\\\\
Thus all intersection curves of $A \cap E$ must be arcs that begin and end on the two different boundary components of $E$. Choose any component $D'$ of $A \cap Y$. Doubling $Y$ across $E$ yields a genus two generalized Kishino knot and $D$ doubles to a vertical annulus, again contradicting Lemma \ref{Kishino2}. So there are no vertical annuli in $R \setminus J'$.  
\\\\
We now show that $R \setminus J'$ is tg-hyperbolic. By  Lemma \ref{lma: sphere-tori}, it is enough to show there are no essential tori.  Suppose $T$ is such an essential torus such that its number of intersection curves with $E$ have been minimized. For the same reasons as argued above, it must intersect $E$ and it can only intersect $E \setminus (E \cap J')$ in nontrivial curves that are also nontrivial on $T$. Let $F$ be an annulus component of $A \cap Y$. Then again, doubling $Y$ across $E$ yields  a tg-hyperbolic genus two generalized Kishino knot and $F$ doubles to a torus $T'$. The only way that $T'$ is not essential is if it is boundary-parallel.
So the components of $T \cap Y$ are all boundary-parallel annuli parallel into $\partial N(J')$.
\\\\
Now let $F'$ be an annulus in $T \cap X$. From what we just said, its boundary components must each bound a disk on $E$ that has one puncture.
If both boundary components of $F'$ circle the same puncture, then we can form a torus in $X$ from $F'$ and the annulus $A''$ bounded by the two boundary components of $F'$ on $E$. This torus also lives in the surface-link pair $(S,K')$ which then means it must compress, and therefore we can isotope $F'$ to $A''$ and then a little farther to lower the number of intersection curves in $T \cap E$, a contradiction.
\\\\
So each of the  boundary components of $F'$ must circle one puncture on $E$. If they circle different punctures, we can extend $F'$ to be a torus in the tg-hyperbolic surface-link pair $(S, K')$. Because it is tg-hyperbolic, the torus must be boundary-parallel, forcing $T$ to be boundary-parallel as well.  Thus, there are no essential tori in $R \setminus J'$ and therefore it is tg-hyperbolic. 
\end{proof}

Note that all of the knots 4.1--4.8, 4.54--4.56, 4.74, 4.76, 4.77 are  half-Kishino appends to projections of either the trivial knot (in which case the theorem does not apply) or the virtual trefoil 2.1.


\begin{proof}[{\bf Proof of Theorem \ref{thm:kishino-hyperbolic}}]
Lemmas \ref{Kishino2} and \ref{kishinoappend} immediately imply that the generalized $n$-Kishino knots are tg-hyperbolic virtual links $K_n$ with $c(K_n) = 2n$ and $g(L_n) = n$.
\end{proof}

\noindent 

\subsection{Minimal crossing genus $g$ tg-hyperbolic knots}

\begin{lemma}\label{thm:min-crossing-number}
For any virtual link $L$, with minimal classical crossing number $c$ and minimal genus $g$,  $c \geq 2g - 1$.
\end{lemma}

\begin{proof}
Every diagram $D$ for $L$ with $c(D)$ classical crossings comes from some representation of the link in the surface $S_{g(D)}$ such that the complement of the projection of the link is a disjoint union of disks. This gives a cell decomposition of the surface of genus $g(D)$ with $c(D)$ $0$-cells, $2c(D)$ 1-cells, and some number $f(D)$ of 2-cells. The Euler characteristic of this cell complex is, on the one hand, $f(D)-c(D)$. On the other hand, it is given by $2-2g(D)$. Since $f(D)$ is at least one, this yields that $c(D) \geq 2g(D) - 1$. Thus, for all diagrams $D$, we have $c(D) \geq 2g - 1$, where $g$ is the minimal genus of the relevant virtual link. Substituting a diagram for $L$ realizing the minimal number of classical crossings in for $D$ yields the desired inequality.

\end{proof}

\begin{theorem}\label{thm:hopfshino}
For all $g > 0$ there exists a tg-hyperbolic virtual knot $K$ with genus $g$ and crossing number $c(K) = 2g - 1$.
\end{theorem}

\begin{proof}
Simply take any of the virtual triple crossing knots 3.1, 3.3 or 3.4, all of which are tg-hyperbolic with genus 2. Then append $g-2$ Kishino appends. The resulting tg-hyperbolic knot has a projection with $2g-1$ crossings that yields a minimal genus representation of genus $g$. Hence, by Lemma \ref{thm:min-crossing-number}, the crossing number must be exactly $2g-1$. 
\end{proof}

Let $v_{oct} \approx 3.6638$ be the volume of an ideal regular octahedron in hyperbolic 3-space.

\begin{corollary}\label{cor:volume-upper-bound}
There exists a genus $g$ virtual knot $K_g$ such that $vol(K_g) \leq (4g -2) v_{oct}$.
\end{corollary}

\begin{proof} This follows from Theorem \ref{thm:hopfshino} and Corollary 4.1 in \cite{adams-meyer-calderon}, which states that the volume of a hyperbolic link $L'$ in $S \times I$ is bounded above by $2 v_{oct} c(L')$. 
\end{proof}

We will use this last fact in the next section.


\section{Minimal Genus and Hyperbolic Volume}\label{hyperbolic-volume}

By the J\o rgensen-Thurston Theorem (see \cite{benedetti-petronio}  (corollary E.7.1 and corollary E.7.5), the set of volumes of orientable hyperbolic $3$-manifolds is well-ordered. Therefore, for every $g > 0$ there is a genus $g$ virtual link $\mathcal{L}_g$ such that for any genus $g$ virtual link $L_g$, $Vol(\mathcal{L}_g) \leq Vol(L_g)$. 

\begin{definition} Define 
$V_{\min}(g)$ to be the least volume of any genus $g$ virtual tg-hyperbolic link. 
\end{definition}

Note that this is well-defined since we showed in Section \ref{kishino} that there is a virtual hyperbolic link for all genera.
The expectation is that these volumes increase with $g$. However, at this time, with the information we have, we can only prove the following.

\begin{theorem}\label{thm:min-volume-increasing}
For all $g > 0$, $V_{\min}(g) < V_{\min}(q)$, for all integers $q\geq 2g$. 
\end{theorem}

\noindent The proof of Theorem \ref{thm:min-volume-increasing} makes use of the following fact, originally proved by Miyamoto, and stated in \cite{yoshida} in this form.

\begin{theorem}\label{thm:miyamoto}
Let $M$ be a hyperbolic manifold with totally geodesic boundary. Then $\text{vol}(M) \geq \frac{v_{oct}}{2}|\chi(\partial M)|$.
\end{theorem}

Miyamoto further proves that the manifolds that realize the lower bound on volume are all compact.  For our purposes, we immediately have:

\begin{corollary}\label{lowervolumeboundgenus}
For a tg-hyperbolic link $L$ embedded in $S_g \times I$,  where $g \geq 2$, $vol(S_g \times I \setminus L) > 2 v_{oct}(g - 1)$.
\end{corollary}

\begin{proof}[Proof of Theorem \ref{thm:min-volume-increasing}]
Corollary \ref{cor:volume-upper-bound} implies that $V_{min}(g) \leq (4g-2) v_{oct}$. However, Corollary \ref{lowervolumeboundgenus} also implies that $(2q -2) v_{oct} < V_{min}(q)$ for all $q\geq 2$. Therefore $V_{min}(g) < V_{min}(q)$ for all $q \geq 2g$. 
\end{proof}

\medskip

\noindent It is possible to determine the minimal volume two-component virtual tg-hyperbolic link. In \cite{yoshida}, the following theorem is proved.

\begin{theorem}\label{thm:yoshida}
A minimal volume orientable hyperbolic $3$-manifold with four cusps has volume $2 v_{oct}$ and is homeomorphic to the $8^4_2$ link complement. 
\end{theorem}

\begin{corollary}
The minimal volume of any two-component virtual tg-hyperbolic link is $2 v_{oct}$.
\end{corollary}


\begin{proof}The $8^4_2$ link is a non-alternating chain link, as pictured in Figure \ref{fig:4_component}. Note that the Hopf link, whose complement is homeomorphic to $T^1 \times (0,1)$, is a sublink of $8^4_2$. Hence, we can consider $8^4_2$ as a link embedded in $T^1 \times (0,1)$. 
\\\\
Since the $8^4_2$ complement is the least volume $4$-cusped hyperbolic $3$-manifold, the corresponding virtual link must be the lowest volume genus one two-component virtual link, as all such virtual links can be represented as complements of four-component links (with the Hopf link as a sublink). By Corollary \ref{lowervolumeboundgenus}, we know there are no links of higher genus with volume this low.
\end{proof}

\begin{figure}[htbp]
    \centering
    \includegraphics[scale=.5]{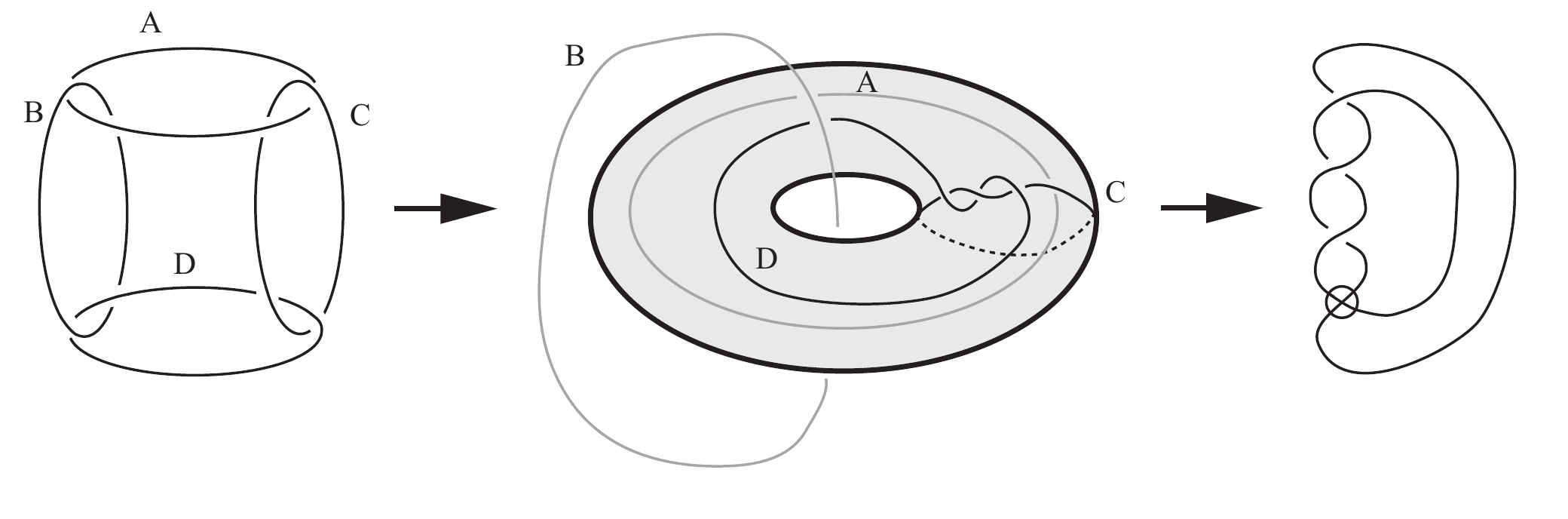}
    \caption{The minimal volume 4-component link that yields the minimal volume 2-component virtual link.}
    \label{fig:4_component}
\end{figure}



\section{Conjectures}\label{conjectures}

\begin{conjecture} The minimal volume non-classical virtual link is the virtual trefoil, with a volume of approximately 5.33349.
\end{conjecture}

By Corollary \ref{lowervolumeboundgenus}, we need only show it is the minimal volume virtual link of genus 1. Yoshida's lower bound on the volume of four-cusped manifold of $2 v_{oct} \approx 7.3223$  also holds for $n$-cusped manifolds for $n \geq 5$, since otherwise, we could do high Dehn filling on all but four cusps to obtain a 4-cusped manifold of smaller volume than Yoshida's bound. Thus, we know that any $n$-component virtual link with $n \geq 2$ has volume at least $2v_{oct}$.

\bigskip

So it is only necessary to show the virtual trefoil has the least volume among virtual knots of genus 1. It is conjectured that the minimally twisted chain link of three components $6_3^3$ is the 3-cusped manifold of minimal volume. If true, this would imply the conjecture, since that link contains a sub-link that is a Hopf link, and hence the link complement is equivalent to a knot complement in $T \times I$. That surface-link pair represents the virtual trefoil.

\begin{conjecture}\label{conj:monotone-increasing}
The function $V_{min}(g)$ is monotone increasing for $g \geq 0$.
\end{conjecture}

We know $V_{min}(0) < V_{min}(1) < V_{min}(2)$. This follows because $V_{min}(0) = 2 v_{tet} \approx  2.0298$, where $v_{tet}\approx 1.01494$ is the volume of an ideal regular hyperbolic tetrahedron. this is true because  in \cite{CM}, $2 v_{tet}$ was proved to be the smallest volume of any cusped orientable hyperbolic 3-manifold and it is realized by the figure-eight knot complement and one other one-cusped hyperbolic 3-manifold. 

\bigskip
    
Moreover, a link in $T \times (0,1)$ is equivalent to a link of three or more components in $S^3$. Since high surgery on all but two of the components yields a 2-cusped hyperbolic manifold of lower volume and since Agol proved in \cite{Agol} that the least volume of a 2-cusped hyperbolic manifold is $v_{oct}\approx 3.6638$, $V_{min}(0) < v_{oct} < V_{min}(1) $.

\bigskip

We also know $V_{min}(1) \leq 5.33349$, coming from the volume of the virtual trefoil knot 2.1 (see table in the next section). 
But from Corollary \ref{lowervolumeboundgenus}, we know $V_{min}(2) > 2v_{oct}$, so we have $V_{min}(1) < V_{min}(2)$.
\\\\
We note that from the table in the next section, $V_{min}(2) \leq 16.392265671$. It is also true that $V_{min}(3) \leq 25.166431466$ by direct calculation of the hyperbolic structure on a genus 3 knot with 5 classical crossings.

\begin{conjecture}For $3\leq n\leq 10$, the minimally twisted $n$-chain link yields the minimal volume of a non-classical virtual link of $n-2$ components. 
\end{conjecture}

In \cite{Agol}, Agol conjectured that the minimal volume link of $n$ components is the minimally twisted $n$-chain link for $3 \leq n \leq 10$. Since each of these links contains a Hopf link, their complement corresponds to the complement of an $(n-2)$-component link in $T \times (0,1)$. Hence, each corresponds to an $(n-2)$-component virtual link. Thus, assuming Agol's conjecture, when restricting to genus one, the result would hold. But we would further have to prove that for higher genus, the $(n-2)$-component virtual links would have greater volume than this.

\bigskip

Note that for $n \geq 11$, the $n$-component link of least known volume is the $(n-1)$-fold cyclic cover over one component of the Whitehead link, with volume $(n-1)v_{oct}$. Since all of these links contain a Hopf sublink, if these links do turn out to be the minimal volume, this would imply that for these values of $n$, the least volume of an $(n-2)$-component genus one virtual link is $(n-1)v_{oct}$. But again, to prove these are the least volume non-classical virtual links with $n-2$ components,  we would further have to prove that for higher genus, the $(n-2)$-component links would have greater volume than this.

\section{Table of Hyperbolic Volumes}\label{table}

This section presents a methodology for computing hyperbolic volumes of links in thickened surfaces, and ends with a table of volumes of virtual knots and links, using the methodology. 
\\\\
The \emph{$S_g \times S_1$ surgery link} $H_g$ is the link depicted in Figure \ref{fig:horrendouslink}. The link has $2g +1$ components, where $g$ is the resulting genus.  Note that the  link $H_1$ is the Borromean rings. 
\begin{figure}[htbp]
    \centering
    \includegraphics[width=0.3\textwidth]{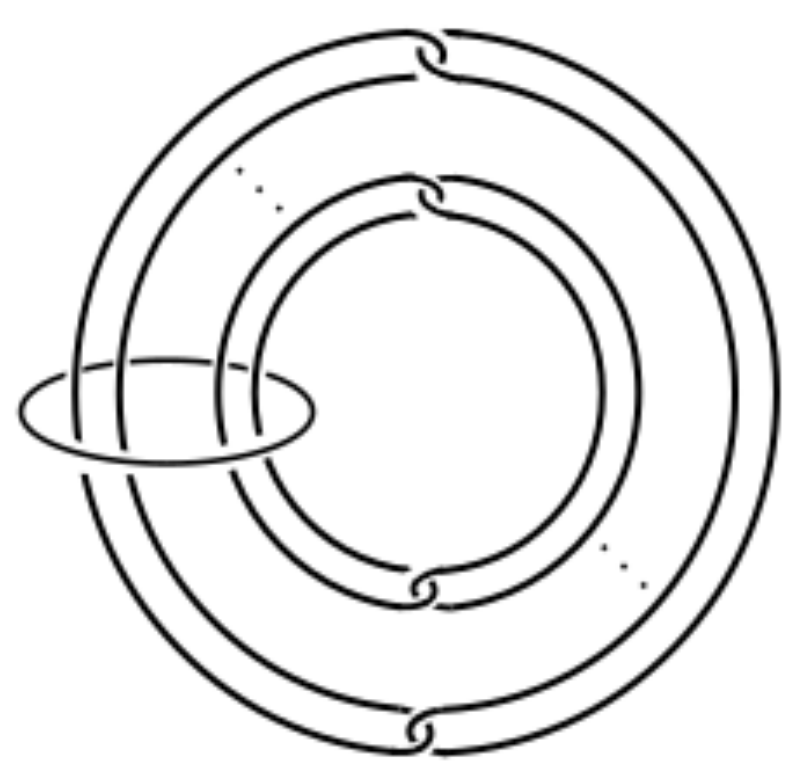}
    \caption{The $S_g \times S_1$ surgery link $H_g$.}
    \label{fig:horrendouslink}
\end{figure}

\noindent The following lemma appears as an exercise in \cite{montesinos}. 
\begin{lemma}
    $S_g\times S^1$ can be obtained from $(0,1)$-Dehn filling on all components of the $S_g \times S_1$ surgery link $H_g$.
\end{lemma}

In order to compute the volume of a genus one link, we noted that the complement of the Hopf link in $S^3$ is $T \times (0,1)$, where $T$ is a torus. So the complement of a link in a thickened torus can be realized as the complement of a link of two more components in $S^3$, adding in the two components corresponding to the Hopf link that generates the $T \times (0,1)$.  Such a link can be input into the hyperbolic structures program SnapPy (\cite{SnapPy}) to compute the volume. 
\\\\
For genus $g \geq 2$, we take advantage of Dehn surgery on the link pictured in Figure \ref{fig:horrendouslink}. Doing (0, 1)-surgery on all of the components of such a link gives $S^g \times S^1$. 
\\\\
For a given virtual tg-hyperbolic link of genus $g \geq 2$ with $t$ components, we can find a tg-hyperbolic representation $(S, L')$. Then we double the manifold $S \times I \setminus L'$ across its two totally geodesic boundaries. This yields $S \times S^1 \setminus L''$ where $L'' $ is the link of $2t$ components that results by taking the union of a copy of $L'$ and a copy of its reflection in $S \times S^1$. The manifold $S \times S^1 \setminus L''$ can be input into the hyperbolic structures program SnapPy (\cite{SnapPy}) by placing $L''$ around the  $S_g \times S_1$ surgery link $H_g$ appropriately, and then doing $(0,1)$-Dehn filling on all of the components of $H_g$. The volume of the original virtual link is then half the volume that results. 
\\\\
There follows a table of hyperbolic volumes of virtual knots through four classical crossings. Knots are labelled as in Jeremy Green's virtual knot table (\cite{green}). The only nontrivial non-hyperbolic knot is the 3.6 knot, which is the classical trefoil.

\begin{center}
    \begin{tabular}{| l | l | l |}
    \hline
    Virtual Knot & Minimal Genus & Hyperbolic Volume \\ \hline
    2.1 & 1 & 5.33348956690 \\ \hline
    3.1 & 2 & 18.7531474071 \\ \hline
    3.2 & 1 & 7.70691180281 \\ \hline
    3.3 & 2 & 16.392265671 (same as 3.4) \\ \hline
    3.4 & 2 & 16.392265671 (same as 3.3) \\ \hline
    3.5 & 1 & 6.3545865572  \\ \hline
    3.6 & 0 & Not hyperbolic \\ \hline
    3.7 & 1 & 9.5034039310 \\ \hline
    4.1 & 2 & 22.493379859 \\ \hline
    4.2 & 2 & 22.611788156 (same as 4.69, 4.76, 4.98) \\ \hline 
    4.3 & 2 & 20.572042198 (same as 4.6) \\ \hline 
    4.4 & 2 & 21.200355468 (same as 4.30) \\ \hline
    4.5 & 2 & 21.9154088043 \\ \hline
    4.6 & 2 & 20.572042198 (same as 4.3) \\ \hline
    4.7 & 2 & 19.127715255 \\ \hline
    4.8 & 2 & 18.831683367 (same as 4.71, 4.77)\\ \hline
    4.9 & 2 & 23.669963712 \\ \hline
    4.10 & 2 & 23.103877032 \\ \hline
    4.11 & 2 & 23.856837980 \\ \hline
    4.12 & 1 & 10.6669791338 (same as 4.53, 4.73, 4.75 and twice 2.1) \\ \hline
    4.13 & 2 & 23.942021763 \\ \hline
    4.14 & 2 & 23.6657445099\\ \hline
    4.15 & 2 & 19.472911306 (same as 4.35) \\ \hline
    4.16 & 2 & 21.447588496 \\ \hline
    4.17 & 2 & 21.273940484 \\ \hline
    4.18 & 2 & 21.4281087698 \\ \hline
    4.19 & 2 & 21.3847454445 \\ \hline
    4.20 & 2 & 17.9102390621 (same as 4.38, 4.50) \\ \hline
    4.21 & 2 & 20.681258869 \\ \hline
    4.22 & 2 & 20.857960899 \\ \hline
    4.23 & 2 & 21.632850855  \\ \hline
    4.24 & 2 & 21.906855990 \\ \hline
    4.25 & 2 & 20.505494782 (same as 4.27) \\ \hline
    4.26 & 2 & 21.301761314 (same as 4.28)\\ \hline
    4.27 & 2 & 20.5054947824 (same as 4.25) \\ \hline
    4.28 & 2 & 21.301761314 (same as 4.26)\\ \hline
    4.29 & 2 & 21.961806381 \\ \hline
    4.30 & 2 & 21.200355468 (same as 4.4) \\ \hline
    4.31 & 2 & 21.8298273713 \\ \hline
    4.32 & 2 & 21.551562929 \\ \hline
    4.33 & 2 & 21.710522971  \\ \hline
    4.34 & 2 & 22.071520116 \\ \hline
    4.35 & 2 & 19.472911306 (same as 4.15) \\ \hline
    4.36 & 1 & 8.7933456039 \\ \hline
    4.37 & 1 & 6.5517432879 \\ \hline
    4.38 & 2 & 17.910239062 (same as 4.20, 4.50) \\ \hline
    4.39 & 2 & 21.135859812 \\ \hline
    4.40 & 2 & 20.273838496 (same as 4.41)\\ \hline
    4.41 & 2 & 20.2738384955 (same as 4.40)\\ \hline
    4.42 & 2 & 21.328567359 \\ \hline
    \end{tabular}
\end{center}
\begin{center}
    \begin{tabular}{| l | l | l |}
    \hline
    Virtual Knot & Minimal Genus & Hyperbolic Volume \\ \hline
    4.43 & 1 & 8.9673608488 \\ \hline
    4.44 & 2 & 21.735522882 (same as 4.46) \\ \hline
    4.45 & 2 & 23.626618267 (same as 4.47) \\ \hline
    4.46 & 2 & 21.735522882 (same as 4.44)  \\ \hline
    4.47 & 2 & 23.6266182673 (same as 4.45) \\ \hline
    4.48 & 2 & 20.618594305 \\ \hline
    4.49 & 2 & 21.2669639843 \\ \hline
    4.50 & 2 & 17.910239062 (same as 4.20, 4.38) \\ \hline
    4.51 & 2 & 21.459015648 (same as 4.85) \\ \hline
    4.52 & 2 & 21.456918384 \\ \hline
    4.53 & 1 & 10.6669791338 (same as 4.12, 4.73,4.75 and twice 2.1) \\ \hline
    4.54 & 2 & 21.839438180 \\ \hline 
    4.55 & 2 & 21.418632337 (\textbf{Kishino knot}) \\ \hline
    4.56 & 2 & 23.773639799 \\ \hline
    4.57 & 2 & 18.9694297800 \\ \hline
    4.58 & 2 & 19.0560288254 \\ \hline
    4.59 & 2 & 18.831683367 (same as 4.8, 4.71, 4.77)\\ \hline
    4.60 & 2 & 19.127715255 (same as 4.89) \\ \hline
    4.61 & 2 & 21.8860118290 (same as 4.68)\\ \hline
    4.62 & 2 & 23.5558331860 \\ \hline
    4.63 & 2 & 20.550520767 \\ \hline
    4.64 & 1 & 8.9293178231 \\ \hline
    4.65 & 1 & 10.5568662552 \\ \hline
    4.66 & 2 & 24.271905954 \\ \hline
    4.67 & 2 & 23.6009265895 \\ \hline
    4.68 & 2 & 21.886011829 (same as 4.61)\\ \hline
    4.69 & 2 & 22.611788156 (same as 4.2,4.76, 4.98)\\ \hline
    4.70 & 2 & 22.5937096032 \\ \hline
    4.71 & 2 & 18.831683367 (same as 4.8, 4.77)\\ \hline
    4.72 & 2 & 22.161104637 \\ \hline
    4.73 & 1 & 10.6669791338 (same as 4.12, 4.53, 4.75 and twice 2.1) \\ \hline
    4.74 & 2 & 21.1567287042 \\ \hline
    4.75 & 1 & 10.6669791338 (same as 4.12, 4.53, 4.73 and twice 2.1)\\ \hline
    4.76 & 2 & 22.6117881556 (same as 4.2, 4.69, 4.98)\\ \hline 
    4.77 & 2 & 18.831833668 (same as 4.8, 4.71)\\ \hline
    4.78 & 2 & 18.665545421 (same as 4.39, 4.79) \\ \hline
    4.79 & 2 & 18.6655454521 (same as 4.39, 4.78) \\ \hline
    4.80 & 2 & 17.593333761 (same as 4.81)\\ \hline
    4.81 & 2 & 17.593333761 (same as 4.80)\\ \hline
    4.82 & 2 & 19.498261755 (same as 4.84)\\ \hline
    4.83 & 2 & 21.2165000389 \\ \hline
    4.84 & 2 & 19.498261755 (same as 4.82)\\ \hline
    4.85 & 2 & 21.459015648 (same as 4.51) \\ \hline
    4.86 & 1 & 8.7786588032 \\ \hline
    4.87 & 2 & 19.0663669549 (same as 4.88, 4.93)\\ \hline
    4.88 & 2 & 19.0663669549 (same as 4.87, 4.93)\\ \hline
    4.89 & 2 & 19.1277152546 (same as 4.60) \\ \hline
    4.90 & 2 & 23.128377627 \\ \hline
    4.91 & 1 & 6.7551948165 \\ \hline
    4.92 & 1 & 7.51768989647 \\ \hline
    4.93 & 2 & 19.0663669549 (same as 4.87, 4.88)\\ \hline
    \end{tabular}
\end{center}
\begin{center}
    \begin{tabular}{| l | l | l |}
    \hline
    Virtual Knot & Minimal Genus & Hyperbolic Volume \\ \hline
    4.94 & 1 & 9.9665118837 \\ \hline
    4.95 & 1 & 10.9616419532 \\ \hline
    4.96 & 2 & 22.698482328 \\ \hline
    4.97 & 2 & 24.053938791 \\ \hline
    4.98 & 2 & 22.611788156 (same as 4.2, 4.69, 4.76) \\ \hline
    4.99 & 1 & 8.7385704088 \\ \hline
    4.100 & 1 & 8.3555021464 \\ \hline
    4.101 & 1 & 12.2113073079 \\ \hline
    4.102 & 1 & 10.714084294 \\ \hline
    4.103 & 2 & 24.604224640 (same as 4.88)\\ \hline
    4.104 & 1 & 11.340719807  \\ \hline
    4.105 & 1 & 11.5188395851 \\ \hline
    4.106 & 1 & 12.0776477618 \\ \hline
    4.107 & 2 & 26.7015415469 \\ \hline
    4.108 & 0 & 2.02988321282\\ \hline
    \end{tabular}
\end{center}

\bibliography{references}{}

\begin{thebibliography}{10}

\bibitem{Adams2}
C.~Adams.
\newblock Augmented alternating link complements are hyperbolic.
\newblock {\em Low-dimensional Topology and Kleinian groups, Cambridge Univ.
  Press}, London Math. Soc. Lecture Note Series. 112, 1986.

\bibitem{SMALL2017}
C.~Adams, C.~Albors-Riera, B.~Haddock, Z.~Li, D.~Nishida, B.~Reinoso, and
  L.~Wang.
\newblock Hyperbolicity of links in thickened surfaces.
\newblock {\em Topology and its Applications}, 256:262--278, 2019.

\bibitem{adams-meyer-calderon}
C.~Adams, A.~Calderon, and N.~Mayer.
\newblock Generalized bipyramids and hyperbolic volumes of alternating
  k-uniform tiling links.
\newblock {\em ArXiv 1709.00432}, 2017.

\bibitem{small18}
C.~Adams, O.~Eisenberg, J.~Greenberg, K.~Kapoor, Z.~Liang, K.~O'Connor,
  N.~Pacheco-Tallaj, and Y.~Wang.
\newblock Turaev hyperbolicity of classical and virtual knots.
\newblock 2018.

\bibitem{AMR}
C.~Adams, W.~Meeks, and A.~Ramos.
\newblock Moves on links preserving hyperbolicity.
\newblock {\em in preparation}, 2019.

\bibitem{Agol}
I.~Agol.
\newblock The minimal volume orientable hyperbolic 2-manifolds of minimal
  volume.
\newblock {\em Proceedings of the AMS}, 138:3723--3732, 2010.

\bibitem{benedetti-petronio}
R.~Benedetti and C.~Petronio.
\newblock {\em Lectures on hyperbolic geometry}.
\newblock Springer, 1992.

\bibitem{CM}
C.~Cao and R.~Meyerhoff.
\newblock The orientable cusped hyperbolic 3-manifolds of minimum volume.
\newblock {\em Invent. Math.}, 146:451--478, 2001.

\bibitem{carter-kamada-saito}
J.~S. Carter, S.~Kamada, and M.~Saito.
\newblock Stable equivalence of knots on surfaces and virtual knot cobordisms.
\newblock {\em Journal of Knot Theory and its Ramifications}, 11(2), 2000.

\bibitem{CKP}
A.~Champanerkar, I.~Kofman, and J.~Purcell.
\newblock Geometry of biperiodic alternating links.
\newblock {\em Journal of Lon. Math. Soc.}, 2018.

\bibitem{SnapPy}
M.~Culler, N.~Dunfield, M.~Goerner, and J.~Weeks.
\newblock Snap{P}y, a computer program for studying the geometry and topology
  of $3$-manifolds.
\newblock Available at \url{http://snappy.computop.org} (2018).

\bibitem{dye-kauffman}
H.A. Dye and L.~H. Kauffman.
\newblock Minimal surface representations of virtual knots and links.
\newblock {\em Algebraic Geometric Topology}, 5:509--535, 2005.

\bibitem{green}
J.~Green.
\newblock {\em A table of virtual knots}.
\newblock https://www.math.toronto.edu/drorbn/Students/GreenJ, 2004.

\bibitem{hatcher}
A.~Hatcher.
\newblock {\em Notes on basic 3-manifold topology}.
\newblock https://pi.math.cornell.edu/~hatcher/3M/3Mfds.pdf, 2007.

\bibitem{HP}
J.~Howie and J.~Purcell.
\newblock Geometry of alternating links on surfaces.
\newblock {\em ArXiv:1712.01373}, 2017.

\bibitem{kauffman}
L.~Kauffman.
\newblock Virtual knot theory.
\newblock {\em Europ. J. Combinatorics}, 20:663--691, 1999.

\bibitem{kuperberg}
G.~Kuperberg.
\newblock What is a virtual link?
\newblock {\em Algebraic Geometric Topology}, 3:583--591, 2003.

\bibitem{menasco}
W.~Menasco.
\newblock Closed incompressible surfaces in alternating knot and link
  complements.
\newblock {\em Topology}, 23:37--44, 1984.

\bibitem{montesinos}
J.~M. Montesinos-Amilibia.
\newblock {\em Classical tessellations and 3-manifolds}.
\newblock Springer, 1987.

\bibitem{thurston}
W.~Thurston.
\newblock {\em The Geometry and Topology of 3-Manifolds}.
\newblock lecture notes, Princeton University, 1978.

\bibitem{yoshida}
K.~Yoshida.
\newblock The minimal volume orientable hyperbolic 3-manifold with 4 cusps.
\newblock {\em Pacific J. of Math.}, 266(2):457--476, 2013.

\end{thebibliography}
\nocite{*}
\bibliographystyle{plain}


\end{document}